\documentclass[12pt,a4paper]{amsart}
\usepackage{graphicx}
\usepackage{amssymb}
\usepackage{amsfonts}
\usepackage{bbm}
\usepackage[matrix,frame]{xypic}
\usepackage{tikz}
\newcommand{\tmmathbf}[1]{\ensuremath{\boldsymbol{#1}}}
\newcommand{\tmop}[1]{\ensuremath{\operatorname{#1}}}

\newcommand{\tmstrong}[1]{\textbf{#1}}

\newtheorem{example}{Example}[section]
\newtheorem{note}[example]{Note}
\newtheorem{theorem}[example]{Theorem}

\newtheorem{corollary}[example]{Corollary}

\newtheorem{definition}[example]{Definition}
\newtheorem{proposition}[example]{Proposition}

\newtheorem{lemma}[example]{Lemma}

\def\Proof{\noindent \it Proof -- \rm}
\def\qed{\hspace{3.5mm} \hfill \vbox{\hrule height 3pt depth 2 pt width 2mm}
\bigskip}

\def\ca{{\rm ca}}
\def\rR{\diamond}
\def\Mould{{\bf Mould}}
\def\std{{\rm std}}
\def\raff{\succeq}
\def\O{{\mathcal O}}
\def\QSym{{\it QSym}}          
\def\NCSF{{\bf Sym}}           
\def\FQSym{{\bf FQSym}}        
\def\MQSym{{\bf MQSym}}        
\def\WQSym{{\bf WQSym}}        
\def\FF{{\mathcal F}}
\def\LL{{\mathfrak g}}
\def\GG{{\mathcal G}}
\def\gautrid{\!\prec\!}   
\def\miltrid{\circ}       
\def\droittrid{\!\succ\!} 

\def\P{{\bf P}}                 %

\def\pack{{\rm pack}}          

\def\bm{\tmmathbf}
\def\ev{{\rm ev}}       
\def\ssh{\Cup}          
\def\saug{\uplus}       
\def\sconc{\bullet}     
\def\Std{{\rm std}}     
\def\<{\langle}
\def\>{\rangle}
\def\NN{{\mathbb N}}    
\def\ZZ{{\mathbb Z}}    
\def\QQ{{\mathbb Q}\, } 
\def\RR{{\mathbb R}}    
\def\CC{{\mathbb C}}    
\def\KK{{\mathbb K}\, } 

\def\park{{\bf a}} 

\def\F{{\bf F}}         
\def\G{{\bf G}}         
\def\SG{{\mathfrak S}}  

\def\maj{{\rm maj}}

\def\Des{\operatorname{Des}}

\def\convW{{*_W}}       

\def\Res{\operatorname{Res}}

\def\PW{{\rm PW}}

\def\shuff#1#2{\mathbin{
\hbox{\vbox{ \hbox{\vrule \hskip#2 \vrule height#1 width 0pt
}%
\hrule}%
\vbox{ \hbox{\vrule \hskip#2 \vrule height#1 width 0pt
\vrule }%
\hrule}%
}}}

\def\qbin#1#2{\begin{bmatrix} #1 \\ #2\end{bmatrix}}

\def\shuf{{\mathchoice{\shuff{7pt}{3.5pt}}%
{\shuff{6pt}{3pt}}%
{\shuff{4pt}{2pt}}%
{\shuff{3pt}{1.5pt}}}}%
\def\shuffle{\,\shuf\,}

\def\park{{\bf a}}    

\def\N{{\bf N}}
\def\M{{\bf M}}       

\def\MM{{\mathcal M}} 
\def\TT{{\mathcal T}} 
\def\PP{{\mathcal P}}

\def\K{{\mathbb K}}   
\def\R{{\mathbb R}}   
\def\Sym{{\bf Sym}}   
\def\fc{{\bf 1}}

\def\gf#1#2{\genfrac{}{}{0pt}{}{#1}{#2}}

\def\twoh{\twoheadrightarrow}
\def\hookr{\hookrightarrow}

\title[]{Mould calculus, polyhedral cones, and characters of combinatorial
Hopf algebras}

\author[F. Menous, J.-C.~Novelli and J.-Y.~Thibon]
{Fr\'ed\'eric Menous, Jean-Christophe Novelli and Jean-Yves Thibon}

\address[Menous]{Laboratoire de Math\'ematiques, B\^atiment 425, Universit\'e
Paris-Sud, 91405 Orsay Cedex, FRANCE}
\email[Fr\'ed\'eric Menous]{frederic.menous@math.u-psud.fr}
\address[Novelli and Thibon] {Universit\'e Paris-Est Marne-la-Vall\'ee \\
Laboratoire d'Infor\-ma\-tique Gaspard-Monge (CNRS - UMR 8049)\\
77454 Marne-la-Vall\'ee cedex 2 \\
FRANCE}
\email[Jean-Christophe Novelli]{novelli@univ-mlv.fr}
\email[Jean-Yves Thibon]{jyt@univ-mlv.fr} 


\date{}

\begin{document}

\begin{abstract}
We describe a method for constructing characters of combinatorial
Hopf algebras by means of integrals over certain polyhedral cones.
This is based on ideas from resurgence theory, in particular on the
construction of well-behaved averages induced by diffusion processes on the
real line.
We give several interpretations and proofs of the main result in terms of
noncommutative symmetric and quasisymmetric functions, as well as
generalizations involving matrix quasi-symmetric functions.
The interpretation of noncommutative symmetric functions as alien operators
in resurgence theory is also discussed, and a new family of Lie idempotents
of descent algebras is derived from this interpretation.
\end{abstract}

\maketitle
{\footnotesize
\tableofcontents
}

\section{Introduction}

The algebra $\Sym$ of noncommutative symmetric functions is a universal object.
Apart from providing noncommutative versions of various identities among
classical symmetric functions \cite{NCSF1}, it can  be interpreted in
representation theory as the Grothendieck ring of the tower of $0$-Hecke
algebras \cite{NCSF4}, or as the direct sum of the descent algebras of
symmetric groups (noncommutative analogues of the character rings)
\cite{NCSF1,MR}. In algebraic topology, it is the homology of the loop space
of the suspension of the infinite dimensional complex projective space
\cite{BR08}. It is also related to the geometry of polytopes \cite{BE}.  In
this context, it is also called the universal Leibniz Hopf algebra \cite{Haz}.
Actually, if one chooses as in \cite{ABS} to define a combinatorial Hopf
algebra as a graded connected Hopf algebra endowed with a character, then it
becomes an initial object in the corresponding category, and its dual $QSym$
(quasi-symmetric functions \cite{Ge}) is a terminal object.

The most interesting applications of noncommutative symmetric functions are
perhaps those related to the descent algebras, and particularly to the
so-called Lie idempotents. These are idempotents of the symmetric group
algebras which act on tensor algebras $T(V)$ as projectors onto the free Lie
algebra $L(V)$.  There is an extensive literature on the subject (see, {\it e.g.},
\cite{BBG,NCSF2}), often related to the Hausdorff series \cite{Dyn,So1} or
various formal expansions for the solutions of differential systems
\cite{MP,Wil}.

However, before the paper \cite{NCSF2}, only three examples of Lie idempotents
were known. These were the Dynkin/Specht/Wever idempotent \cite{Dyn,Spe,Wev},
the Solomon/Eulerian idempotent \cite{So1} (both related to the Hausdorff
series), and the Klyachko idempotent \cite{Kl}, which arose as an intertwining
operator between two representations of $GL(V)$.  The contribution of
\cite{NCSF2} was to show that all these examples were specializations of
natural families of noncommutative symmetric functions with one or more
parameters, and that their properties could be derived from noncommutative
analogues of classical manipulations with ordinary symmetric functions.  In
particular, the properties of the Klyachko idempotent appeared as the simplest
analogue of a specialization property of Hall-Littlewood functions at roots of
unity.  Also, Lie idempotents of descent algebras were characterized as the
(suitably normalized) primitive elements with a nonzero commutative image.

This provided a potentially infinite supply of Lie idempotents, and it was
tempting to assume that any new example that could arise would be readily
interpreted in terms of the constructions of \cite{NCSF2}.

This belief was to be turned down by a recently discovered interpretation of
$\Sym$.
On the occasion of a one-year seminar on Ecalle's mould calculus during the
academic year 2007-2008, it was realized that the algebra of noncommutative
symmetric functions (and some of its generalizations) played an important role
in the theory of resurgent functions. It arises in the guise of a Hopf algebra
of alien operators, which act on certain function spaces by means of intricate
combinations of analytic continuations. Among those are the alien derivations,
which turn out to correspond to Lie idempotents in $\Sym$.

The point of view of resurgence and mould calculus leads to insights different
from those from the theory of noncommutative symmetric functions.  For
example, a technique for constructing moulds with prescribed symmetries from
random walks on the real line leads to remarkable examples of Lie idempotents
given by explicit formulas.

The aim of the present paper is to explain these connections, and to try to
unify both points of view. We shall see in particular that the construction of
alien automorphisms (grouplike series in $\Sym$) by means of random walks can
be traced back to a geometric property of certain polyhedral cones, which
allows to interpret the calculation in other combinatorial Hopf algebras like
$\WQSym$ or $\MQSym$.

After recalling the relevant properties of noncommutative symmetric functions
(Section \ref{sec:NCSF}), we present in Section \ref{moyennes} 
the first version of our main result: grouplike and primitive series of
noncommutative symmetric functions can be constructed by means of certain
iterated integrals. To understand this property, we have to embed $\Sym$ in a
larger algebra, $\WQSym$, whose main properties are recalled in Section
\ref{moreCHA}. This algebra can be regarded as based on set compositions. In
Section \ref{cones}, we associate two polyhedral cones with a set composition,
and obtain our main result as a consequence of a geometric property:
the characteristic function of a Cartesian product of cones is an alternating
sum of characteristic functions of similar cones. Our first proof relies upon
certain multivariate Laurent series, the so-called integer point transforms of
the cones. These series represent rational functions in their domain of
convergence, and given such a function together with the corresponding domain,
the cone can be reconstructed unambiguously.  This representation of elements
of a combinatorial Hopf algebra by rational functions is reminescent of that
used in \cite{CHNT}, although of a different nature. The exact analogue of the
constructions of \cite{CHNT} are given in Section \ref{ratMould}, where
nonlinear operators  associated with elements of $\WQSym$ are constructed by
means of discrete iterated integrals. As in \cite{CHNT}, an operadic
interpretation is provided, and the tridendriform operad is related to
$\WQSym$ in a simple and explicit way.
This interpretation can also give rise to characters of $\WQSym$, and we
indicate briefly how to recover some familiar examples.
Next, the constructions of Section~\ref{sec:MQSym} are extended to the Hopf
algebra $\MQSym$, which may be interpreted as based on multiset compositions.
This allows to define more general polyhedral cones, to which the main result
is extended in Section \ref{sec:newcones}. In this context, it can be proved
by a short (but tricky) algebraic calculation, whose meaning is eventually
interpreted in terms of Rota-Baxter algebras (Section \ref{secRB}).
In Section \ref{catalan}, we present a class of iterated integrals which can
be evaluated in closed form, and obtain a new family of Lie idempotents, the
Catalan family.  Finally, we review the connections between noncommutative
symmetric functions and alien calculus (Section \ref{sec:alien}), and sketch
the proof of the isomorphism of Hopf algebras between noncommutative symmetric
functions and alien operators (Section \ref{sec:coprod}).

\section{Noncommutative symmetric functions}
\label{sec:NCSF}

\subsection{The Hopf algebra $\Sym$}

By definition, $\Sym$ is the free associative algebra $\K\<S_1,S_2,\dots\>$
over an infinite sequence $S_n$, endowed with the grading $\deg S_n=n$ and the
coproduct
\begin{equation}
\Delta S_n=\sum_{k=0}^nS_k\otimes S_{n-k} \quad (S_0:=1)\,,
\end{equation}
where $\K$ is a field of characteristic $0$.

It can be realized in terms of polynomials\footnote{By ``polynomials'', we
mean formal series of finite degree.}
over an auxiliary set $A=\{a_1<a_2<\dots\}$ of noncommuting variables
endowed with a total order.
If $t$ is another indeterminate, commuting with the $a_i$, we set
\begin{equation}
\sigma_t(A) :=
\prod^\rightarrow_{i\ge 1}(1-ta_i)^{-1}=
(1-ta_1)^{-1}(1-ta_2)^{-1}\dots
=\sum_{n\ge 0}S_n(A)t^n \quad 
\end{equation}
(the arrow means that the product should be taken from left to right)
so that
\begin{equation}
\lambda_{-t}(A) :=
\prod^\leftarrow_{1\le i}(1-ta_i)
= \dots(1-ta_2)(1-ta_1)
=\sum_{n\ge 0}\Lambda_n(A)(-t)^n
=\sigma_t(A)^{-1}\,.
\end{equation}
Then the coproduct can be expressed as
\begin{equation}
\Delta F=F(A+B)
\end{equation}
where $B=\{b_i|i\ge 1\}$ is another ordered alphabet isomorphic to $A$ and
where $A+B$ is interpreted as the {\em ordinal sum} of $A$ and $B$
(a noncommutative sum!), and $A$ commutes with $B$ for the multiplication.

As a Hopf algebra, $\Sym$ is not self-dual. This is clear from the definition
of the coproduct, which is obviously cocommutative. Its (graded) dual is the
commutative algebra $QSym$ of {\em quasi-symmetric functions} \cite{Ge,NCSF1}.

\subsection{Bases}
\label{bases}

From the generators $S_n$, we can form a linear basis
\begin{equation}
S^I = S_{i_1}S_{i_2}\dots S_{i_r}
\end{equation}
of the homogeneous component $\Sym_n$, parametrized by {\em compositions}
of $n$, that is, finite ordered sequences $I=(i_1,\dots,i_r)$ of positive
integers summing to $n$. One often writes $I\vDash n$ to mean that $I$
is a composition of $n$. The dimension of $\Sym_n$ is $2^{n-1}$
for $n\ge 1$.

Similarly, from the $\Lambda_n$ we can build a basis 
\begin{equation}
\Lambda^I = \Lambda_{i_1}\Lambda_{i_2}\dots \Lambda_{i_r}.
\end{equation}
We can also look for analogues of the power-sum symmetric functions. It is
here that the non-trivial questions arise. Indeed, in the commutative case,
the power-sum $p_n$ is, up to a scalar factor, the unique primitive element of
degree $n$.
Here, the primitive elements form a free Lie algebra, and it is not
immediately obvious to identify the ones which should deserve the name
``noncommutative power sums''.

However, we can at least give one example: the series $\sigma_t$ being
grouplike ($\Delta\sigma_t=\sigma_t\otimes\sigma_t$), its logarithm is
primitive, and writing
\begin{equation}
\log \sigma_t =\sum_{n\ge 1}\Phi_n \frac{t^n}{n}\,, 
\end{equation}
we can reasonably interpret $\Phi_n$ as a noncommutative power-sum. 

Finally, let us present a last basis of $\NCSF$ which is the analog of the
Schur functions in this context.
In terms of the auxiliary ordered alphabet $A$
\begin{equation}
S_n(A)=\sum_{i_1\le i_2\le\dots\le i_n}a_{i_1}a_{i_2}\dots a_{i_n}\,,
\end{equation}
is the sum of nondecreasing words.
Let us say that a word $w=a_{i_1}a_{i_2}\dots a_{i_n}$
has a {\em descent} at $k$ if $i_k>i_{k+1}$, and denote by $\Des(w)$
the set of such $k$ (the {\em descent set} of~$w$).

Thus, $S_n(A)$ is the sum of words of length $n$ with no descent. 
Now, obviously,
\begin{equation}
S^I= S_{i_1}S_{i_2}\dots S_{i_r}
\end{equation}
is the sum of words whose descent set is contained in 
\begin{equation}
\{i_1,i_1+i_2,\dots,i_1+i_2+\dots+i_{r-1}\}\,.
\end{equation}
We denote this set by $\Des(I)$ and call it the
\emph{descent set of the composition} $I$.
Symmetrically, we call $I$ the \emph{descent composition}, and write $I=C(w)$,
of any word of length $n$ having $\Des(I)$ as descent set.

The \emph{noncommutative ribbon Schur functions} are defined by
\begin{equation}
R_I(A)=\sum_{C(w)=I}w
\end{equation}
so that we have
\begin{equation}
S^I =\sum_{J\le I}R_J
\end{equation}
where $J\le I$ is the \emph{reverse refinement order}, which means
that $\Des(J)\subseteq\Des(I)$.

The product in the ribbon basis is given by
\begin{equation}\label{prodR}
R_I\cdot R_J= R_{I\cdot J}+R_{I\triangleright J}
\end{equation}
where 
\begin{equation}
I\cdot J=(i_1,\dots,i_r,j_1,\dots,j_s)\  \text{and}\
I\triangleright J=(i_1,\dots,i_{r-1},i_r+j_1,j_2,\dots,j_s)
\end{equation}
(note that
this formula is  obvious from the interpretation in terms of words).
For example, 
\begin{equation}\label{RRIJ}
R_{132}R_2=R_{1322}+R_{134}\,.
\end{equation}

\subsection{Descent algebras}
Let $(W,S)$ be a Coxeter system. One says that $w\in W$
has a descent at $s\in S$ if $w$ has a reduced word ending by $s$.
For $W=\SG_n$ and $s_i=(i,i+1)$, this means that $w(i)>w(i+1)$,
whence the terminology. In this case, we rather say
that $i$ is a descent of $w$.
Let  $\Des(w)$ denote the descent set of $w$, and for a subset
$E\subseteq S$, set
\begin{equation}
D_E = \sum_{\Des(w)=E} w \ \ \in \ZZ W \,.
\end{equation}
Solomon has shown \cite{Sol} that the $D_E$ span a $\ZZ$-subalgebra
$\Sigma(W)$ of $\ZZ W$. Moreover
\begin{equation}
D_{E'}D_{E''}=\sum_E c^E_{E'E''} D_E
\end{equation}
where the coefficients $c^E_{E'E''}$ are nonnegative integers.

In the case of $W=\SG_n$, we encode descent sets by compositions of $n$ as
explained above.
If $E=\{d_1,\dots,d_{r-1}\}$,
we set $d_0=0$, $d_r=n$ and $I=C(E)=(i_1,\dots,i_r)$, where $i_k=d_k-d_{k-1}$.
From now on, we shall write $D_I$ instead of $D_E$, and denote by
$\Sigma_n$ the descent algebra of $\SG_n$ (with coefficients in our
ground field $\K$).

Thus, $\Sigma_n$ has the same dimension as $\Sym_n$, and both have
natural bases labelled by compositions of $n$.

The natural correspondence $\Sigma_n\rightarrow \Sym_n$ is defined by the
linear map $\alpha:\, D_I \rightarrow R_I$.  
It allows to transport the product of the descent algebra to
$\Sym$. The resulting operation is the internal product $*$.
We set $F*G=0$ if $F$ and $G$ are homogeneous of different degrees, and
for technical reasons, we want our correspondence to be an
{\em anti-isomorphism}:
\begin{equation}
R_I* R_J= \alpha(D_JD_I)\,.
\end{equation}

If we map our letters $a_i$ to commuting variables $x_i$,
noncommutative symmetric functions are mapped to ordinary symmetric
functions. For example, the $\Lambda_n$ go to the familiar elementary
symmetric functions $e_n$ and the $\Phi_n$ to the power-sums
$p_n=\sum_ix_i^n$.

It is proved in \cite{NCSF1,NCSF2} that a homogeneous element
$\Pi_n\in\Sym_n$ is the image under $\alpha$ of a Lie idempotent
of the descent algebra $\Sigma_n$ if and only if it is primitive and
has as commutative image $\frac1n p_n$.

For example, $\frac1n\Phi_n$ corresponds to the Solomon (Eulerian) idempotent \cite{So1},
whilst $\Psi_n$ defined by the recurrence
\begin{equation}\label{eq:newtonpsi}
nS_n = S_{n-1}\Psi_1+S_{n-2}\Psi_2+\dots+\Psi_n 
\end{equation}
corresponds to the Dynkin operator
\begin{equation}
a_1a_2\dots a_n \mapsto [\dots[[a_1,a_2],a_3],\dots a_n]
\end{equation}
and is given by
\begin{equation}\label{eq:pnhook}
\Psi_n = R_n - R_{1,n-1}+ R_{1,1,n-2}- \dots
= \sum_{k=0}^{n-1}(-1)^k R_{1^k,n-k}\,.
\end{equation}

A nontrivial example is the one-parameter family of Lie idempotents
\begin{equation}
\label{EULID}
\varphi_n(q) =
\frac{1}{n} 
\sum_{\sigma\in\SG_n} \frac{(-1)^{d(\sigma)}}{\qbin{n-1}{ d(\sigma)}_q} \
q^{\maj(\sigma) - \binom{d(\sigma)+1}{2}} \,\sigma,
\end{equation}
where the major index
$\maj(\sigma)$ of a permutation is the sum of its descents.
One can prove that this family
interpolates between the Dynkin and Solomon idempotents
and gives back the Klyachko idempotent
when $q$ is a primitive $n$th root of unity (see \cite{NCSF2}). 
Further properties
of this idempotent, including a preLie expression, can be found in \cite{Cha2}.

\section{Iterated integrals and characters}
\label{moyennes}

We define in this section some special elements of $\Sym$ corresponding
to operators
originally introduced by Ecalle in the framework of real resummation and
alien calculus. For the sake of simplicity, we will stay within the
formalism of noncommutative symmetric functions and the link with
resummation will be explained at the end of this paper.

For a permutation $\sigma\in\SG_n$, consider the sequence of $\pm$
signs 
\begin{equation}
\mathbf{\varepsilon}
  \bullet=(\varepsilon_1,\dots,\varepsilon_{n-1},\bullet),
\end{equation}
where $\varepsilon_i$ is the sign of $\sigma(i+1)-\sigma(i)$. Note
that the symbol $\bullet$ is added so that the length of the sequence
is $n$. With this notation, it is clear that
\begin{equation}
\Des(\sigma)=\{1\leq i \leq n-1 \ ;\ \varepsilon_i=-\}.
\end{equation}
Let us now introduce the {\it signed ribbon basis} of $\Sym$,
 which is a slight modification
of the noncommutative ribbon Schur functions:
\begin{equation}
R_{\mathbf{\varepsilon} \bullet}=(-1)^{l(I)-1}R_I\quad
(R_\emptyset=1,\ R_{\bullet}=R_1),
\end{equation}
where $I$ is the composition such that
\begin{equation}
\Des(I)=\{1\leq i \leq n-1 \ ;\ \varepsilon_i=-\}.
\end{equation}
So, Equation (\ref{prodR}) reads
\begin{equation}
R_{\mathbf{a}\bullet}R_{\mathbf{b}\bullet}
= R_{\mathbf{a}+\mathbf{b}\bullet}-R_{\mathbf{a}-\mathbf{b}\bullet}.
\end{equation}
For example (compare (\ref{RRIJ})),
\begin{equation}\label{RR+-}
R_{-++-+\bullet}R_{+\bullet}=R_{-++-+++\bullet}-R_{-++-+-+\bullet}\,.
\end{equation}
Let us also define
\begin{equation}
\mathcal{E} = \bigcup_{n \geq 0} \mathcal{E}_n = \bigcup_{n \geq 0} \{
\tmmathbf{\varepsilon} = (\varepsilon_1, \dots, \varepsilon_n) \hspace{1em} ;
\hspace{1em} \varepsilon_i = \pm\}.
\end{equation} 

Our main goal here is to study the elements of $\Sym$ whose
coefficients in the basis $R_{\mathbf{\varepsilon} \bullet}$ 
are defined as probabilities.

Let $X_1, \dots, X_n$ be real, independent, identically distributed random
variables of density $f$, and define
\begin{equation}
m_f^{\varepsilon_1,\dots,\varepsilon_n}
 = P(\varepsilon_1 S_1>0,\dots, \varepsilon_n S_n >0),
\end{equation}
with $S_k=X_1+\dots+X_k$. These ``weights'' (whose collection is
called an average induced by diffusion in \cite{em1}) will allow us to
find new grouplike and primitive elements in $\Sym$. Their
structure does not depend on the fact that $f$ is a probability
density, since these probabilities can be defined as ``iterated
integrals'' of $f$ on some subsets of $\RR^n$, and make sense for
any bounded integrable function.

\begin{definition}
\label{defmf}   
We denote by $\fc_X$ the characteristic function of a subset $X$ of $\RR^n$, 
and set
$\sigma_+ =\tmmathbf{1}_{\RR^+}$,
and $\sigma_- =\tmmathbf{1}_{\RR^{- \ast}}$. 

Let $f\in \mathcal{A}=L^1 (\RR) \cap L^{\infty}(\RR)$ be a bounded integrable
function.
We set  $m^\emptyset=1$, $d^\emptyset=0$ and, for a nonempty sequence
of signs
$\tmmathbf{\varepsilon}= (\varepsilon_1 ,\dots ,\varepsilon_n)\in \mathcal{E}$, 
\begin{equation}
\label{eq-defmf}
m_f^{\tmmathbf{\varepsilon}} = \int_{\RR^n} f (x_1) \dots f (x_n)
   \sigma_{\varepsilon_1} (x_1) \sigma_{\varepsilon_2} (x_1 + x_2) \dots
   \sigma_{\varepsilon_n} (x_1 + \dots + x_n) d x_1 \dots d x_n.
\end{equation}
In the same way, for any sequence of signs
$\tmmathbf{\varepsilon}= (\varepsilon_1 ,\dots ,\varepsilon_n)\in \mathcal{E}$
of length at least $2$, we set
\begin{equation}\label{deps}
\begin{split}
d_f^{\tmmathbf{\varepsilon}}
& = \varepsilon_n \int_{\RR^n} f (x_1) \dots f (x_n) \\
&\ \ \sigma_{\varepsilon_1} (x_1) \sigma_{\varepsilon_2} (x_1 +
   x_2) \dots \sigma_{\varepsilon_{n - 1}} (x_1 + \dots + x_{n - 1})
   \delta (x_1 + \dots + x_n) dx_1\dots d x_n.
\end{split}
\end{equation}
where $\delta$ is the Dirac distribution concentrated at $0$
(\emph{i.e.}, $\int_X\delta(x)dx$ is 1 or 0 according to whether $X$ contains
$0$ or not).
\end{definition}

Thanks to the regularity of $f$, these integrals are well-defined and when $f$
is continuous at $0$, Equation~(\ref{deps}) still makes sense for sequences of
length $1$ and we set $d^+=-d^-=f(0)$. Otherwise, we can give any arbitrary
value to $d^+=-d^-$. We then have

\begin{theorem}
\label{thmRm}
Define formal series $R_f$, $L_f$ and $D_f$ of noncommutative symmetric
functions by
\begin{equation}
\begin{split}
R_f &= 1+\sum_{\tmmathbf{\varepsilon}\in \mathcal{E}}
       R_f^{\tmmathbf{\varepsilon} \bullet}
       R_{\tmmathbf{\varepsilon} \bullet},\\
L_f &= 1+\sum_{\tmmathbf{\varepsilon}\in \mathcal{E}}
       L_f^{\tmmathbf{\varepsilon} \bullet}
       R_{\tmmathbf{\varepsilon} \bullet},\\
D_f &= \sum_{\tmmathbf{\varepsilon}\in \mathcal{E}}
       D_f^{\tmmathbf{\varepsilon} \bullet}
       R_{\tmmathbf{\varepsilon} \bullet},
\end{split}
\end{equation}
where
\begin{equation}
\begin{array}{cccccc}
R_f^{\tmmathbf{\varepsilon} \bullet} & = & m_f^{\tmmathbf{\varepsilon}+} & 
&  & (R^{\emptyset} = 1),\\
L_f^{\tmmathbf{\varepsilon} \bullet} & = & - m_f^{\tmmathbf{\varepsilon}-}
&  &  & (L^{\emptyset} = 1),\\
D_f^{\tmmathbf{\varepsilon} \bullet} & = & d_f^{\tmmathbf{\varepsilon}+} &
= & - d_f^{\tmmathbf{\varepsilon}-} & (D^{\emptyset} = 0).
\end{array}
\end{equation}
Then $R_f$ and $L_f$ are grouplike and $D_f$ is primitive.
\end{theorem}

\Proof
The only difficulty is to prove that $R_f$ is grouplike.
Then, the corresponding property of $L_f$ is implied by
Proposition~\ref{propL}, and the primitivity of $D_f$ follows from
Lemma~\ref{lemD}.

Different proofs and generalizations of the fact that $R_f$ is grouplike are
given in the sequel.
We shall first obtain it as a corollary of Proposition~\ref{propCu},
which is itself equivalent to Theorem \ref{th1Ku}. This result
provides in fact a general method of constructing characters of
a certain Hopf algebra $\WQSym$, whose definition is recalled in
Section~\ref{moreCHA}.
Then a different method of constructing characters of $\WQSym$
is presented in Section~\ref{ratMould}. Finally, Theorem \ref{thMQSym}
generalizes Theorem \ref{th1Ku} to the  algebra $\MQSym$. In this context,
the result can be proved by a compact algebraic calculation, whose
meaning is then traced back to the theory of Rota-Baxter algebras
(Section \ref{secRB}).

\begin{proposition}
\label{propL}
If $I(f)=\int f(x) dx$, then
\begin{equation}
\label{RL}
R_f=L_f\cdot \sigma_{I(f)},
\end{equation}
where $\sigma_z=\sum_{n\ge 0} z^n S_n$ is the generating series
of complete noncommutative symmetric functions.
\end{proposition}

\begin{proof} 
This result is a consequence of the simple equality
\begin{equation}\label{consist}
m_f^{\mathbf{\varepsilon}+}+ m_f^{\mathbf{\varepsilon}-}=
m_f^{\mathbf{\varepsilon}}I(f)\quad
\text{for all}\ \mathbf{\varepsilon} \in \mathcal{E}. 
\end{equation}
Let $\mathcal{E}_-$ be the subset of sequences of $\mathcal{E}$ that
do not end with a $+$ sign (including the empty sequence $\emptyset$). 
Thanks to Equation (\ref{consist}), for $\mathbf{\eta}\in \mathcal{E}_-$ and
$k\geq 0$, 
\begin{equation}
m_f^{\mathbf{\eta} +^{k+1}}
= m_f^{\mathbf{\eta}}\ I(f)^{k+1}
  - \sum_{i=0}^km_f^{\mathbf{\eta} +^{k-i}-}\ I(f)^{i}
\end{equation}
and
\begin{equation}
\begin{array}{rcl}
R_f &=& 1 + \displaystyle \sum_{\mathbf{\varepsilon}\in \mathcal{E}}
m_f^{\mathbf{\varepsilon}+} R_{\mathbf{\varepsilon}\bullet} \\
&=& 1 + \displaystyle \sum_{\mathbf{\eta}\in \mathcal{E}_-}\sum_{k\geq 0}
m_f^{\mathbf{\eta} +^{k+1}} R_{\mathbf{\eta} +^{k}\bullet} \\
&=& 1+ \displaystyle \sum_{\mathbf{\eta}\in \mathcal{E}_-}\sum_{k\geq
  0}m_f^{\mathbf{\eta}}\ I(f)^{k+1}R_{\mathbf{\eta} +^{k}\bullet} - \sum_{\mathbf{\eta}\in \mathcal{E}_-}\sum_{k\geq
  0}\sum_{i=0}^k m_f^{\mathbf{\eta} +^{k-i}-}I(f)^{i}R_{\mathbf{\eta}
  +^{k}\bullet} \\
&=& 1+ \displaystyle \sum_{\mathbf{\eta}\in \mathcal{E}_-}\sum_{k\geq
  0}m_f^{\mathbf{\eta}}\ I(f)^{k+1}R_{\mathbf{\eta} +^{k}\bullet}- \sum_{\mathbf{\eta}\in \mathcal{E}_-}\sum_{i\geq
  0}\sum_{j\geq 0} m_f^{\mathbf{\eta} +^{j}-}I(f)^{i}R_{\mathbf{\eta}
  +^{i+j}\bullet} \\
&=& 1+ \displaystyle \sum_{\mathbf{\eta}\in \mathcal{E}_-}\sum_{k\geq
  0}m_f^{\mathbf{\eta}}\ I(f)^{k+1}R_{\mathbf{\eta} +^{k}\bullet}-
\sum_{\mathbf{\varepsilon}\in \mathcal{E}}\sum_{i\geq 
  0} m_f^{\mathbf{\varepsilon} -}I(f)^{i}R_{\mathbf{\varepsilon}
  +^{i}\bullet}\,. \\
\end{array}
\end{equation}

Splitting the sum over $\eta$ into two parts, according to whether
$\eta=\emptyset$ or $\eta\in\mathcal {E}-\backslash\{\emptyset\}$, in which
case we write $\eta=\epsilon-$, we have
\begin{equation}
\begin{array}{rcl}
R_f&=& 1+ \displaystyle \sum_{k\geq
  0} I(f)^{k+1}R_{+^{k}\bullet}+\sum_{\mathbf{\varepsilon}\in
  \mathcal{E}}\sum_{k\geq 
  0}m_f^{\mathbf{\varepsilon}-}\ I(f)^{k+1}R_{\mathbf{\varepsilon}
  -+^{k}\bullet} \\
&& \quad \displaystyle - \sum_{\mathbf{\varepsilon}\in \mathcal{E}}\sum_{i\geq 
  0} m_f^{\mathbf{\varepsilon} -}I(f)^{i}R_{\mathbf{\varepsilon}
  +^{i}\bullet}\\
&=& 1+ \displaystyle \sum_{k\geq
  1} I(f)^{k}R_{+^{k-1}\bullet}  +\sum_{\mathbf{\varepsilon}\in
  \mathcal{E}} \sum_{k\geq 
  1} - m_f^{\mathbf{\varepsilon} -} I(f)^{k}(R_{\mathbf{\varepsilon}
  ++^{k-1}\bullet}-R_{\mathbf{\varepsilon}
  -+^{k-1}\bullet}) \\
&& \displaystyle + \sum_{\mathbf{\varepsilon}\in \mathcal{E}}
   - m_f^{\mathbf{\varepsilon} -}R_{\mathbf{\varepsilon}\bullet}\,.
\end{array}  
\end{equation}

But $S_k=R_k=R_{+^{k-1}\bullet}$ and 
\begin{equation}
R_{\mathbf{\varepsilon}
  ++^{k-1}\bullet}-R_{\mathbf{\varepsilon}
  -+^{k-1}\bullet}=R_{\mathbf{\varepsilon}\bullet}R_{+^{k-1}\bullet},
\end{equation}
so that 
\begin{equation}
R_f = 1 + \sum_{k\ge 1}I(f)^k S_k+ 
\sum_{\mathbf{\varepsilon}\in \mathcal{E}}
 - m_f^{\mathbf{\varepsilon}- }R_{\mathbf{\varepsilon}\bullet}
\left( 1+\sum_{k\ge 1}I(f)^kS_k\right),
\end{equation}
and finally
$R_f=L_f\cdot \sigma_{I(f)}$.
\end{proof}

Since $\sigma_z$ is grouplike, we can state:

\begin{lemma}
$R_f$ is grouplike if and only if $L_f$ is grouplike. 
\end{lemma}
\qed

The case of $D_f$ is less obvious. Given a function $f\in \mathcal{A}$, we
define a one-parameter family of functions by
\begin{equation}
 \quad f_t(x)=f(x-t) \quad (t\in\RR).
\end{equation}
When $f\in \mathcal{C}^{\infty}_0\subset \mathcal{A}$ is an infinitely
differentiable function with compact support, the weights
$m_{f_t}^{\mathbf{\varepsilon}}$ are differentiable with respect to $t$,
and
\begin{proposition}
If $f\in \mathcal{C}^{\infty}_0$,
\begin{equation}
\partial_t R_{f_t}=(YD_{f_t})R_{f_t}
\end{equation}
where $Y$ is the Euler operator on $\Sym$ (that is, $YR_I=|I|R_I$).
\end{proposition}

\begin{proof}
For given functions $f_1,\dots,f_n,\dots$ in $\mathcal{A}$, we define
by induction
\begin{equation}
\left [\begin{array}{c}
\emptyset \\
\emptyset
\end{array} \right ] (x)=\delta\,,
\end{equation}

\begin{equation}
\left [\begin{array}{ccc}
\varepsilon_1  &\dots&\varepsilon_n \\
f_1&\dots&f_n
\end{array} \right ] (x)=\left (\left [\begin{array}{ccc}
\varepsilon_1  &\dots&\varepsilon_{n-1} \\
f_1&\dots&f_{n-1}
\end{array} \right ]\ast f_n \right )(x)\sigma_{\varepsilon_n}(x),
\end{equation}
where $*$ is the convolution on $\RR$
\begin{equation}
(f*g)(x)=\int_\RR f(x-y)g(y)dy.
\end{equation}
Definition \ref{defmf} reads now
\begin{equation}
m_f^{\varepsilon_1,\dots, \varepsilon_n}
=\int_{\RR}
 \left [\begin{array}{ccc} \varepsilon_1  &\dots&\varepsilon_n \\
  f&\dots&f
  \end{array} \right ] (x) dx,
\end{equation}
and, since $f\in \mathcal{C}^{\infty}_0$ and $n\geq 1$,
\begin{equation}
\begin{array}{rcl}
d_f^{\varepsilon_1,\dots, \varepsilon_n}
&=& \varepsilon_n
  \left (\left [\begin{array}{ccc}
  \varepsilon_1  &\dots&\varepsilon_{n-1} \\
  f&\dots&f
  \end{array} \right ] * f \right )(0) \\[12pt]
&=& \displaystyle \varepsilon_n \int_{\RR} \left [\begin{array}{ccc}
  \varepsilon_1  &\dots&\varepsilon_{n-1} \\
  f_1&\dots&f_{n-1}
  \end{array} \right ](y)f(-y) dy \\[12pt]
&=& \displaystyle - \int_{\RR} \left [\begin{array}{ccc}
  \varepsilon_1  &\dots&\varepsilon_{n-1} \\
  f_1&\dots&f_{n-1}
  \end{array} \right ](y)
  \left (\varepsilon_n \int_0^{\varepsilon_n \infty}f'(x-y)dx \right ) dy
  \\[12pt]
&=& \displaystyle - \int_{\RR} \left [\begin{array}{ccc}
  \varepsilon_1  &\dots&\varepsilon_{n-1} \\
  f_1&\dots&f_{n-1}
  \end{array} \right ](y)
  \left (\int_{\RR} f'(x-y)\sigma_{\varepsilon_n}(x) dx \right) dy \\[12pt]
&=& \displaystyle -\int_{\RR} \left [\begin{array}{ccccc}
  \varepsilon_1  &\dots&\varepsilon_{n-1}&\varepsilon_n \\
  f&\dots&f&f'
  \end{array} \right ] (x) dx,
\end{array}
\end{equation}
where the last identity comes from the theorem of Fubini.

Using the recursive definition of $\left [\begin{array}{ccc}
\varepsilon_1  &\dots&\varepsilon_n \\
f_t&\dots&f_t
\end{array} \right ]$ and integration by parts, we get
\begin{equation}\label{diffrec}
\begin{split}
\partial_t 
\left[
\begin{array}{ccc}
\varepsilon_1  &\dots&\varepsilon_n \\
f_t&\dots&f_t
\end{array} 
\right] 
=\sum_{k=1}^{n-1} kd_{f_t}^{\varepsilon_1,\dots, \varepsilon_k}& \left [ \begin{array}{ccc}
\varepsilon_{k+1}  &\dots&\varepsilon_{n} \\
f_t&\dots&f_t
\end{array} \right ]\\ 
&+n \left [\begin{array}{ccccc}
\varepsilon_1  &\dots&\varepsilon_{n-1}&\varepsilon_n \\
f_t&\dots&f_t&   \partial_tf_t      
\end{array} \right ]\,.
\end{split}
\end{equation}
Indeed, this equation is obvious for $n=1$ and, recursively,
if \eqref{diffrec} holds for a given $n\geq 1$, let 
\begin{equation}
g(t,x)=\left[
\begin{array}{ccc}
\varepsilon_1  &\dots&\varepsilon_n \\
f_t&\dots&f_t
\end{array} 
\right],
\end{equation}
 then
\begin{eqnarray}
\partial_t 
\left[
\begin{array}{cccc}
\varepsilon_1  &\dots&\varepsilon_n &\varepsilon_{n+1}\\
f_t&\dots&f_t & f_t
\end{array} \right]
&=& \sigma_{\varepsilon_{n+1}}(x)
    \int_{\RR} \partial_t(g(t,y)\ast f_t(x-y)) dy\nonumber \\
&=& \sigma_{\varepsilon_{n+1}}(x)
    \int_{\RR} (\partial_t g(t,y)).f_t(x-y) dy \\
& &   +\,\sigma_{\varepsilon_{n+1}}(x)
    \int_{\RR} g(t,y) (\partial_t f_t(x-y)) dy\nonumber
\end{eqnarray}
but 
\begin{equation}
\sigma_{\varepsilon_{n+1}}(x) \int_{\RR} g(t,y) (\partial_t f_t(x-y)) dy
= \left[
\begin{array}{cccc}
\varepsilon_1  &\dots&\varepsilon_n &\varepsilon_{n+1}\\
f_t&\dots&f_t & \partial_t f_t
\end{array} 
\right],
\end{equation}
and if we use (\ref{diffrec}) to expand $\partial_t g(t,y)$, then
\begin{eqnarray}
\sigma_{\varepsilon_{n+1}}(x)
 \int_{\RR} (\partial_t g(t,y)).f_t(x-y) dy
&=& \sum_{k=1}^{n-1} kd_{f_t}^{\varepsilon_1,\dots, \varepsilon_k}
 \left [ \begin{array}{cccc}
  \varepsilon_{k+1}  &\dots &\varepsilon_{n} &\varepsilon_{n+1}\\
  f_t&\dots&f_t &f_t
 \end{array} \right ]\nonumber \\
& & + n \sigma_{\varepsilon_{n+1}}(x)
 \left [\begin{array}{ccccc}
   \varepsilon_1  &\dots & \varepsilon_{n-1}&\varepsilon_n \\
   f_t&\dots&f_t&   \partial_tf_t      
  \end{array} \right ]\ast f_t(x).
\end{eqnarray}
The last term reads 
\begin{equation}
n \sigma_{\varepsilon_{n+1}}(x)\left (\left (\left [\begin{array}{cccc}
\varepsilon_1  &\dots & \varepsilon_{n-1}&\\
f_t&\dots&f_t    
\end{array} \right ]
\ast \partial_t f_t \right)\sigma_{\varepsilon_{n}}\right )\ast f_t  (x)
\end{equation}
and if 
\begin{equation}
h(x)=\left [\begin{array}{cccc}
\varepsilon_1  &\dots & \varepsilon_{n-1}&\\
f_t&\dots&f_t    
\end{array} \right ](x), 
\end{equation}
then,
\begin{equation}
\begin{split}
((h*\partial_t f_t)\sigma_{\varepsilon_{n}})\ast f_t (x)
&=-\int_{\RR} \int_{\RR} h(z)f'(y-z-t)
  \sigma_{\varepsilon_{n}}(y)f(x-y-t) dz dy\nonumber  \\
&= -\int_{\RR} \varepsilon_n
  \int_0^{\varepsilon_n \infty}f'(y-z-t)f(x-y-t)dy \,h(z) dz\nonumber\\
&= \varepsilon_n
  \int_{\RR} \Big( f(-z-t)f(x-t) \\
& \qquad\qquad -
  \int_0^{\varepsilon_n \infty}f(y-z-t)f'(x-y-t)dy\Big) h(z) dz\nonumber\\
&= \varepsilon_n (h\ast f_t)(0) f_t(x)+((h\ast f_t)\sigma_{\varepsilon_n})\ast
\partial_t f_t(x).
\end{split}
\end{equation}
This yields the required equality at order $n+1$.

Integrating these equations (with $\partial_t f_t(x)=-f'(x-t)$), we obtain
\begin{equation}
\partial_t m_{f_t}^{\varepsilon_1 \dots \varepsilon_n}=\sum_{k=1}^n k\, d_{f_t}^{\varepsilon_1 \dots \varepsilon_k}m_{f_t}^{\varepsilon_{k+1}  \dots\varepsilon_{n}}
\end{equation}
and finally
\begin{equation}\label{eqdiffR}
\partial_t R_{f_t}=(YD_{f_t})R_{f_t}\,.
\end{equation}
\end{proof} 

This identity implies the primitivity of $D_f$. Indeed, if $R_f$ is grouplike
for all $f$, in particular, for a given $f$, all the $R_{f_t}$ are also
grouplike, and we have:

\begin{lemma}
\label{lemD}
If $R_{f_t}$ is grouplike for all $t$,
then $D_{f_t}$ is primitive for all $t$.
\end{lemma}

\Proof Let us set for short $R(t)=R_{f_t}$ and $D(t)=D_{f_t}$ and first assume
that $f\in \mathcal{C}^{\infty}_0\subset \mathcal{A}$ is an infinitely
differentiable function with compact support.
Since the coproduct of $\Sym$ has the form $\Delta F=F(A+B)$,
Equation \eqref{eqdiffR} holds for the coproducts, and we have
\begin{equation}
\label{eqYD}
\Delta (YD(t)) = (\Delta R(t))'\Delta(R(t)^{-1})\,.
\end{equation}
Assuming that $R(t)$ is grouplike, we have 
\begin{equation}
(\Delta R(t))'=R'(t)\otimes R(t) + R(t)\otimes R'(t).
\end{equation}
Substituting this expression into \eqref{eqYD}, we obtain
\begin{equation}\label{Ycop}
\Delta (YD(t)) = YD(t)\otimes 1 +1\otimes YD(t)\,.
\end{equation}
Thus $YD(t)$ is primitive, and so are its homogeneous components,
so that $D(t)$ is primitive as well. In the general case $f\in \mathcal{A}$, since $\mathcal{C}^{\infty}_0$ is dense in $L^1(\RR)$
(\ref{Ycop}) still holds for $f\in \mathcal{A}$
\qed

Let us finally remark that we also have integral expressions of $R_f$
and $D_f$ in the basis $\Lambda^I$ (defined in Section \ref{bases}).
Let us denote, for any element $F$ in $\Sym$, its coefficients in a basis $B$
by  $\<F,B\>_I$, so that in particular
\begin{equation}
F=\sum \<F,\Lambda\>_I \Lambda^I\,.
\end{equation}

\begin{proposition}
For $I=(i_1,..,i_r)\vDash n$,
\begin{equation}
\label{mouldR}
\begin{array}{rcl}
\<R_f,\Lambda\>_I &=& \displaystyle (-1)^{r+n}\int_{\RR} \left
    [ \begin{array}{ccc}
+ &\dots &+ \\
f^{*i_1} & \dots & f^{*i_r}
\end{array} \right ](x) dx \\
&=&\displaystyle (-1)^{r+n}\int_{K_I} f(x_1)\dots f(x_n)
dx_1\dots dx_n,
\end{array}
\end{equation}
where $f^{*i}$ is the $i$-th convolution power (on $\RR$) of $f$ and $K_I$ is
the polyhedral cone (see Section \ref{cones}) in $\RR^n$ defined by the
inequalities
\begin{equation}
\quad \sum_{k=1}^{i_1+\dots+i_q} x_k \geq 0 \ \text{for $1\leq q\leq r$}.
\end{equation}
Similarly,
\begin{equation}
\label{mouldD}
\<D_f,\Lambda\>_I =  (-1)^{r+n} \left (\left
    [ \begin{array}{ccc}
+ &\dots &+ \\
f^{*i_1} & \dots & f^{*i_{r-1}}
\end{array} \right ]* f^{*i_{r}} \right )(0) .
\end{equation}
\end{proposition}

\Proof
The right-hand side of Equation \eqref{mouldR} can be rewritten as
\begin{equation}
\begin{split}
(-1)^{r+n}
\sum_{\gf{\varepsilon_{i_1}=+,}{\varepsilon_{i_1+i_2}=+,\dots,\varepsilon_n=+}}
& \int_{\RR}\left[\begin{array}{ccc} \varepsilon_1  &\dots&\varepsilon_n \\
f&\dots&f
\end{array} \right] (x) dx\\
&= (-1)^{\ell(I)+n}\sum_{J\le \bar I^\sim}(-1)^{\ell(J)-1}\<R_f,R\>_J\,.
\end{split} 
\end{equation}
Thus, the result follows from the relation \cite[Eq. (63)]{NCSF1}
\begin{equation}
R_I=\sum_{J\le \bar I^\sim}(-1)^{\ell(J)-\ell(\bar I^\sim)}\Lambda^J,
\end{equation}
since $\ell(\bar I^\sim)=n+1-\ell(I)$.

Similarly, the right-hand side of \eqref{mouldD} is
\begin{equation}
 (-1)^{\ell(I)+n}\sum_{J\le \bar I^\sim}(-1)^{\ell(J)-1}\<D_f,R\>_J\,.
\end{equation}
\qed

It remains to prove that $R_f$ is grouplike. This will be done in the
following sections and the different proofs give rise to some
remarkable interpretations and developments based on these ``iterated
integrals''. We will also emphasize the case of the Catalan average, that
leads to an explicit Lie idempotent.
To go ahead, we first need to introduce a few more algebraic structures.

\section{More combinatorial Hopf algebras}
\label{moreCHA}

\subsection{Quasi-symmetric functions}

As already mentioned, noncommutative symmetric functions and quasi-symmetric
functions can help to understand certain properties of ordinary symmetric
functions. In the same way, certain features of $\Sym$ or $QSym$ can be properly
understood by introducing bigger Hopf algebras, of which they are subalgebras or quotients.

The shortest way to introduce $QSym$ and its duality with $\Sym$ is via the noncommutative
Cauchy identity. Let $X=\{x_1<x_2<\dots\}$ be an infinite totally ordered set of
mutually commuting variables, also commuting with the $a_i$ of $\Sym$. The Cauchy
kernel is the formal series
\begin{equation}
K(X,A) =\prod_{i\ge 1}^\rightarrow \prod_{j\ge 1}^\rightarrow (1-x_ia_j)^{-1},
\end{equation}   
where the arrow means that the products are taken in increasing order from
left to right.
It is easy to expand the product on the basis $S^I(A)$ of $\Sym$:
\begin{equation}
K(X,A) =\prod_{i\ge 1}^\rightarrow \sigma_{x_i}(A)
=\prod_{i\ge 1}^\rightarrow \sum_{j_i\ge 0}x_i^{j_i}S^{j_i}(A)
=\sum_I M_J(X)S^J(A),
\end{equation}
where
\begin{equation}
M_J(X)=\sum_{i_1<i_2<\dots<i_r}x_{i_1}^{j_1}x_{i_2}^{j_2}\dots x_{i_r}^{j_r}\,.
\end{equation}
The polynomials $M_J(X)$ are called monomial quasi-symmetric functions (or quasi-monomial functions).
They span a subalgebra of $\K[X]$, which is precisely $QSym$ \cite{Ge}. 
The bilinear map from $QSym\times\Sym$ to
$\K$ defined by
\begin{equation}
\<M_I,\,S^J\> = \delta_{IJ}
\end{equation}
realizes $QSym$ as the (graded) dual of $\Sym$, and any pair of bases such that
\begin{equation}
K(X,A)=\sum_I U_I(X) V_J(A)
\end{equation}
are dual to each other. The dual of the ribbon basis $R_I$ is the fundamental basis $F_I$,
which has the explicit expression
\begin{equation}\label{defFI}
F_J(X)=\sum_{\substack{i_1\le i_2\le \dots\le i_n \\
                       i_k<i_{k+1} \text{if $k\in\Des(J)$}}}
x_{i_1}x_{i_2}\dots x_{i_n}=\sum_{I\ge J}M_I
\,.
\end{equation}

\subsection{Moulds and their fundamental symmetries}\label{secmould}

The series $R_f$ is grouplike if and only if the coefficients $\<R_f,S\>_I$
define a character, that is, a linear form $\chi$ on $QSym$ satisfying
$\chi(xy)=\chi(x)\chi(y)$, with $\<R_f,S\>_I=\chi(M_I)$.
Similarly,
the series $D_f$ is primitive if and only if the coefficients $\<D_f,S\>_I $
define an infinitesimal character, that is, a linear form $\psi$ on $QSym$
satisfying $\psi(xy)=\psi(x)\varepsilon(y)+\varepsilon(x)\psi(y)$,
with $\<D_f,S\>_I=\psi(M_I)$, where $\varepsilon$ is the counit.

The product rule for the monomial basis is easily derived directly
by duality. It is given by the quasi-shuffle of compositions, regarded as
words over the alphabet of positive integers:
\begin{equation}
M_I M_J = \sum_{K}(K|I\saug J)M_K
\end{equation}
where $(K|I\saug J)$ means the coefficient of the word $K$ in the
quasi-shuffle of the words $I$ and $J$. The quasi-shuffle makes sense for
words over an arbitrary additive semigroup $\Sigma$. It is recursively defined
by
\begin{equation}
\label{eq-shuf-aug}
au\saug bv = a (u\saug bv) + b(au\saug v) +(a+b)(u\saug v)
\end{equation}
where $u,v$ are arbitrary words over $\Sigma$, and $a,b\in\Sigma$,
and the condition that the empty word is neutral.

For example,
\begin{equation}
\begin{split}
13\saug 32 &= 1332+1332+1323+3132+3123+3213\\
&\ \ +162+432+423+135+315+333+45.
\end{split}
\end{equation}

The product formula for the basis $M_I$ allows to identify $QSym$ with
the quasishuffle Hopf algebra over the additive semigroup of positive integers
$\KK\<\NN^*\>$ (see \cite{Hof}). 
We identify $M_I$, $I=(i_1, \dots,  i_r)$ with the
basis $i_1  \dots  i_r$ of $\KK\<\NN^*\>$ (words over the alphabet of positive
integers), equipped with the quasishuffle product and the deconcatenation
coproduct.

Families of coefficients such as $\<R_f,S\>_I$ or $\<D_f,S\>_I$,
which define linear maps from $\KK\<\NN^*\>$ to the base field $\KK$, appear
in Ecalle's work and are called {\it moulds}.
A mould is said to be {\it symmetrel} (resp. {\it alternel}) if and only if it
defines a character (resp. an infinitesimal character) of the quasishuffle
Hopf algebra $\KK\<\NN^*\>$.

Equivalently, moulds can be interpreted as nonlinear operators on $\Sym$.
By definition, $\Sym$ is a graded free associative algebra, with exactly
one generator for each degree. 
Several sequences of generators are of common use, some of which being
composed of primitive elements (such as $\Psi_n$ or $\Phi_n$), whilst other
are sequences of divided powers (such as $S_n$ or $\Lambda_n$), so that their
generating series is grouplike. Each pair of such sequences $(U_n)$, $(V_n)$
defines two moulds, whose coefficients express the expansions of the $V_n$ on
the $U^I$, and vice-versa. These moulds can be interpreted as the
automorphisms sending $U_n$ to $V_n$ or conversely.

Ecalle's four fundamental symmetries reflect the four possible combinations of
the primitive or grouplike characteristics.

If we denote by $\LL$ the (completed) primitive Lie algebra of $\Sym$ and by
$\GG=\exp\,\LL$ the associated  multiplicative group, we have the following
table
$$
\begin{tabular}{|c|c|}
\hline
$\LL \rightarrow \LL$ & \text{Alternal}\\
\hline
$\LL \rightarrow \GG$ & \text{Symmetral}\\
\hline
$\GG \rightarrow \LL$ & \text{Alternel}\\
\hline
$\GG \rightarrow \GG$ & \text{Symmetrel}\\
\hline
\end{tabular} 
$$

\subsection{Noncommutative quasi-symmetric functions: $\FQSym$ and $\WQSym$}

\subsubsection{Free quasi-symmetric functions}
The multiplicative structure of $QSym$ in the bases $F_I$ and $M_I$ can be
understood by lifting these to two different combinatorial Hopf algebras,
which could both equally deserve the name ``noncommutative quasi-symmetric
functions''. For this reason, the first one is called ``Free Quasi-Symmetric
functions'' ($\FQSym$) and the second one ``Word Quasi-Symmetric functions''
($\WQSym$).

To understand the origin of the first one, recall that the noncommutative
ribbon Schur function $R_I$ has two interpretations:

(i) as the sum of words of shape $I$ in the free associative algebra, and

(ii) as the sum of permutations of shape $I$ in the group algebra of
the symmetric group.

So, one may ask whether it is possible to associate with each word of shape
$I$ a permutation of shape $I$, so as to reconcile both approaches, and
interpret each permutation as a sum of words.

This is indeed possible, and the solution is given by the classical
{\em standardization process}, familiar in combinatorics and in computer
science.

The \emph{standardized word} $\Std(w)$ of a word $w\in A^*$ is the permutation
obtained by iteratively scanning $w$ from left to right, and labelling
$1,2,\dots$ the occurrences of its smallest letter, then numbering the
occurrences of the next one, and so on. Alternatively, $\sigma=\Std(w)^{-1}$
can be characterized as the unique permutation of minimal length such that
$w\sigma$ is a nondecreasing word. For example, $\Std(bbacab)=341625$.

Obviously, $\std(w)$ has the same descents as $w$. We can now define
polynomials
\begin{equation}
\G_{\sigma}(A) := \sum_{\std(w) = \sigma} w\,.
\end{equation}
It is not hard to check that the linear span of these polynomials is a
subalgebra of $\KK\<A\>$, denoted by $\FQSym(A)$, an acronym for {\em Free
Quasi-Symmetric functions} \cite{NCSF6}.

Since the definition of the $ \G_{\sigma}(A)$ involves only a totally ordered
alphabet $A$, we can apply it to an ordinal sum $A+B$, and as in the case of
$\Sym$, this defines a coproduct if we assume that $A$ commutes with $B$.
Clearly, this coproduct is coassociative and multiplicative, so that we
have a graded (and connected) bialgebra, hence again a Hopf algebra.
It is isomorphic (as a Hopf algebra) to the convolution algebra of permutations 
of Malvenuto-Reutenauer \cite{MR} (see \cite{NCSF6}).

By definition,
\begin{equation}
R_I(A)=\sum_{C(\sigma)=I} \G_{\sigma}(A) 
\end{equation}
so that $\Sym$ is embedded in $\FQSym$ as a Hopf subalgebra.

It is also easy to check that $\FQSym$ is self-dual. If we set
$\F_\sigma=\G_{\sigma^{-1}}$ and
$\<\F_\sigma\,,\,\G_\tau\>=\delta_{\sigma,\tau}$,
then $\<FG,H\>=\<F\otimes G,\Delta H\>$.

Since the graded dual of $\Sym$ is the commutative algebra $QSym$, 
we have a surjective homomorphism $\FQSym^*\twoh QSym$. Its description
is particularly simple: it consists in replacing our noncommuting variables
$a_i$ by commuting ones $x_i$. Then, $\F_\sigma(X)$ depends only on the
descent composition $I=C(\sigma)$, and is equal to the quasi-symmetric
function $F_I(X)$.

Hence, the multiplication rule for the $\F_\sigma$ describes in particular
that of the $F_I$. To state it, we need the following notation.

For a word $w$ on the alphabet $\{1,2,\dots\}$, we denote by $w[k]$ the word
obtained by replacing each letter $i$ by the integer $i+k$.
If $u$ and $v$ are two words, with $u$ of length $k$, one defines
the \emph{shifted concatenation}
$u\sconc v = u\cdot (v[k])$
and the \emph{shifted shuffle}
$ u\ssh v= u\shuffle (v[k])$,
where $\shuffle$ is the usual shuffle product,
defined for words  over an arbitrary
alphabet $A$ by 
\begin{equation}
au\shuffle bv = a(u\shuffle bv)+ b(au\shuffle v),
\end{equation}
where $u,v$ are arbitrary words over $\Sigma$, and $a,b\in\Sigma$,
and the condition that the empty word is neutral
(compare\eqref{eq-shuf-aug}).

Then, the product rule is
\begin{equation}\label{prodF}
\F_\alpha \F_\beta = \sum_{\gamma\in\alpha\ssh\beta}\F_\gamma\,.
\end{equation}

\subsubsection{Word quasi-symmetric functions}

Although it is possible to lift the monomial basis to $\FQSym$, the resulting
polynomials are not positive sums of monomials. To lift the product formula to
a multiplicity-free product of nonnegative polynomials, one has to introduce
the larger algebra $\WQSym$ \cite{Hiv} (which contains $\FQSym$, see, {\it
e.g.}, \cite{NTtri}).
Its definition is similar to that of $\FQSym$. The only difference is that
standardization is replaced by a finer invariant, the \emph{packed word}.

The \emph{packed word} $u=\pack(w)$ associated with a word $w\in A^*$ is
obtained by the following process. If $b_1<b_2<\dots <b_r$ are the letters
occuring in $w$, $u$ is the image of $w$ by the homomorphism $b_i\mapsto a_i$.
For example, $\pack(64661812)= 43441512$.
A word $u$ is said to be \emph{packed} if $\pack(u)=u$. We denote by $\PW$ the
set of packed words.
With such a word, we associate the polynomial
\begin{equation}
\M_u :=\sum_{\pack(w)=u}w\,.
\end{equation}
Under the abelianization
$\chi:\ \K\langle A\rangle\rightarrow\K[X]$, the $\M_u$ are mapped to the monomial
quasi-symmetric functions $M_I$,
$I=\ev(u)=(|u|_a)_{a\in A}$ being the evaluation vector of $u$, that is,
the sequence whose $i$-th term is the number of times the letter $a_i$ occurs in $w$.

These polynomials span a subalgebra 
of $\K\langle A\rangle$, called $\WQSym$ for Word
Quasi-Symmetric functions~\cite{Hiv}. It is a Hopf algebra
for the usual coproduct $A\mapsto A+B$.

\begin{proposition}
The product on $\WQSym$ is given by
\begin{equation} 
\label{prodG-wq}
\M_{u'} \M_{u''} = \sum_{u \in u'\convW u''} \M_u\,,
\end{equation}
where the \emph{convolution} $u'\convW u''$ of two packed words
is defined as
\begin{equation} 
u'\convW u'' = \sum_{v,w ;
u=v\cdot w\,\in\,\PW, \pack(v)=u', \pack(w)=u''} u\,.
\end{equation}
\end{proposition}
For example,
\begin{equation}
\M_{11} \M_{21} =
\M_{1121} + \M_{1132} + \M_{2221} + \M_{2231} + \M_{3321}.
\end{equation}

There is also a basis $\Phi_u$  of $\WQSym$ which is a lift of the fundamental basis 
$F_I$ of $QSym$, in the sense that under abelianization,
$\Phi_u(X)=F_I(X)$, where $I=\ev(u)$. It is defined as follows\footnote{This basis
is different from the basis ${\bf Q}$ of \cite{BZ}.}.

The refinement  order can be extended to packed words \cite{BZ,NTtri}.
We say that $w$ is finer than $w'$, and  write $w\raff w'$, iff
$w$ and $w'$ have same standardized word and the evaluation of $w$ is finer
than the evaluation of $w'$.
Then,
\begin{equation}
\label{Phi2M} 
\Phi_u := \sum_{v ; v\raff u} \M_v.
\end{equation}

Packed words can be naturally identified with \emph{ordered set partitions},
also called \emph{set compositions},
the letter $a_i$ at the $j$th position meaning that $j$ belongs to block $i$.
For example,
\begin{equation}
u=313144132 \ \leftrightarrow\ \Pi=(\{2,4,7\},\{9\},\{1,3,8\},\{5,6\})\,.
\end{equation}
As set composition $\Pi$ can be represented by
\emph{segmented permutation}, that is, a permutation 
obtained by reading  the blocks of
$\Pi$ in increasing order and inserting bars $|$ between blocks.
To avoid confusion, we shall always write segmented permutations between
parentheses.

For example,
\begin{equation}
\label{init-setpart}
\Pi=(\{2,4,7\},\{9\},\{1,3,8\},\{5,6\})\,\
\leftrightarrow (247|9|138|56),
\end{equation}
and we have, in both notations,
\begin{equation}
\Phi_{133142} = \M_{133142} + \M_{134152} + \M_{144253} + \M_{145263}.
\end{equation}
\begin{equation}
\Phi_{(14|6|23|5)}
 = \M_{(14|6|23|5)} + \M_{(14|6|2|3|5)} + \M_{(1|4|6|23|5)} +
   \M_{(1|4|6|2|3|5)}.
\end{equation}

Since $(\Phi_u)$ is triangular over  $(\M_u)$, it is a basis of
$\WQSym$. Note that the order used for summation is a restriction
of the refinement order on compositions, so is a boolean lattice.
Hence, denoting by $\max(w)$ the greatest letter of a word $w$,
\begin{equation}
\label{M2Phi}
\M_u = \sum_{v; v\raff u} (-1)^{\max(v)-\max(u)} \Phi_v.
\end{equation}
For example,
\begin{equation}
\M_{133142} = \Phi_{133142} - \Phi_{134152} - \Phi_{144253} + \Phi_{145263}.
\end{equation}

By construction, the basis $\Phi$ satisfies a product formula  similar to
that of Gessel's basis $F_I$ of $\QSym$ (whence the choice of notation).
To state it, we need an analogue of the shifted shuffle, defined on the special
class of segmented permutations encoding set compositions.

The \emph{shifted shuffle} $\alpha\ssh\beta$ of two such segmented
permutations is obtained from the usual shifted shuffle $\sigma\ssh\tau$ of
the underlying permutations $\sigma$ and $\tau$ by inserting bars
\begin{itemize}
\item between each pairs of letters coming from the same word if they were
separated by a bar in this word,
\item after each element of $\beta$ followed by an element of $\alpha$.
\end{itemize}

For example,
\begin{equation}
(2|1) \ssh (12)
 = (2|134) + (23|14) + (234|1) + (3|2|14) + (3|24|1) + (34|2|1),
\end{equation}
\begin{equation}
(1|2) \ssh (12)
 = (1|234) + (13|24) + (134|2) + (3|1|24) + (3|14|2) + (34|1|2).
\end{equation}

We then have \cite{NTtri}:

\begin{theorem}
The product and coproduct in the basis $\Phi$ are given by
\begin{equation}
\Phi_{\sigma'} \Phi_{\sigma''} = \sum_{\sigma\in \sigma'\ssh\sigma''}
\Phi_{\sigma}.
\end{equation}
\begin{equation}
\Delta\Phi_\sigma =
\sum_{\sigma'|\sigma''=\sigma \text{\ or\ } \sigma'\cdot  \sigma''=\sigma}
\Phi_{\Std(\sigma')} \otimes \Phi_{\Std(\sigma'')}.
\end{equation}
\end{theorem}

For example, in both encodings, we have
\begin{equation}
\begin{split}
& \Phi_{1} \Phi_{121} = \Phi_{1121} + \Phi_{2132} + \Phi_{2121} + \Phi_{3121}.
\\
& \Phi_{(1)} \Phi_{(13|2)}
  = \Phi_{(124|3)} + \Phi_{(2|14|3)} + \Phi_{(24|13)} + \Phi_{(24|3|1)}.
\end{split}
\end{equation}
\begin{equation}
\begin{split}
\Phi_{1312} \Phi_{21} =& \ \ \ \
 \Phi_{131221} + \Phi_{131231} + \Phi_{131232} + \Phi_{131243} +
 \Phi_{141232} \\ &+
 \Phi_{141321} + \Phi_{142321} + \Phi_{142331} + \Phi_{142341} +
 \Phi_{153421} \\ &+
 \Phi_{242321} + \Phi_{242331} + \Phi_{242341} + \Phi_{253421} +
 \Phi_{353421}.
\end{split}
\end{equation}

\begin{equation}
\Delta\Phi_{23121} =
1\otimes\Phi_{23121} + \Phi_{1}\otimes\Phi_{2321} +
\Phi_{11}\otimes\Phi_{121} + \Phi_{211}\otimes\Phi_{21} +
\Phi_{2121}\otimes\Phi_{1} + \Phi_{23121}\otimes1.
\end{equation}

Finally, the Hopf epimorphism $\WQSym\twoh QSym$ (commutative image)
gives rise by duality to a Hopf embedding $\Sym\hookr \WQSym^*$,
which is given by
\begin{equation}\label{SinW}
S^I\mapsto \sum_{\ev(u)=I}\N_u
\end{equation}
where $\N_u$ is the dual basis of $\M_u$.

\section{Polyhedral cones associated with packed words}\label{cones}

We are now in a position to explain the geometric significance of the averages
induced by diffusion. Identifying the sign sequences ${\bm{\varepsilon}}$ with
compositions $I$ and the basis
$R_{\varepsilon_1\dots,\varepsilon_{n-1}\bullet}$ with $\pm R_I$, we have to
understand why the coefficients $R_f^{\bm{\varepsilon\bullet}}$ defined in
Theorem~\ref{thmRm} are of the  form $\pm\chi(M_I)$, for a character $\chi$
of $QSym$.

We shall see that this is the reflect of a geometric property of certain
polyhedral cones associated with packed words.

\subsection{Cones associated with set compositions}

Let $u$ be a packed word of length~$n$ and
\begin{equation}
\Pi(u)= (B_1,\dots,B_r),\quad B_k=\{b_{k,1},\dots,b_{k,i_1}\}
\end{equation}
be the set composition of $[n]$ encoded by $u$
and let $\sigma$ the corresponding segmented permutation,
so that for $u=322123$, $\Pi(u)=(\{4\},\{2,3,5\},\{1,6\})$,
and $\sigma=(4|235|16)$.

\begin{definition}
The polyhedral cone $K_u$ in $\R^n$ is defined by the inequalities
\begin{equation}
\sum_{j=1}^k \sum_{i\in B_j} x_i \ge 0\quad 
\text{for $k=1,\dots,r$}\,.
\end{equation}
\end{definition}

For example, with $u=322123$ as above, $K_u$ is defined by the system
\begin{equation}
\left\{
\begin{matrix}
x_4 & \ge 0, \\
x_4+x_2+x_3+x_5 & \ge 0,\\
x_4 + x_2+x_3+x_5+x_1+x_6 &\ge 0\,.
\end{matrix}
\right.
\end{equation}

We denote by $\fc_S$ the characteristic function of a subset $S$ of $\R^n$.

Let $\FF_n$ be the space of classes of measurable functions
$\RR^n\rightarrow\RR$, where two functions differing on a set of measure zero
are identified.
Define an associative product on 
\begin{equation}
\FF=\bigoplus_{n\ge 0}\FF_n
\end{equation} 
by
\begin{equation}
(f\star g)\, (x_1,\dots,x_{m+n})
 = f(x_1,\dots,x_m)\, g(x_{m+1},\dots,x_{m+n}),
\end{equation}
for $f\in\FF_m$ and $g\in\FF_n$.

Let $\PP$ be the subalgebra of $\FF$ generated by the characteristic
functions $\fc_{K_u}$ of the cones $K_u$.

\begin{theorem}
\label{th1Ku}
The map $\alpha:\ \PP\rightarrow \WQSym$ defined by
\begin{equation}
\fc_{K_u}\mapsto (-1)^{\max(u)}\M_u
\end{equation}
is an isomorphism of algebras. That is,
if the product $\M_u\M_v$ in
$\WQSym$ is given by
\begin{equation}
\M_u\M_v = \sum_w c_{uv}^w\M_w\,,
\end{equation}
the characteristic function of the Cartesian product $K_u\times K_v$
is
\begin{equation}
\label{prodKu}
\fc_{K_u\times K_v}
 = \sum_w  (-1)^{\max(u)+\max(v)-\max(w)} c_{uv}^w\fc_{K_w}\,.
\end{equation}
\end{theorem}

The rest of this section is devoted to the proof of this result.
We shall first translate it into an identity involving characteristic
functions of Cartesian products of different cones, and relate it to the
product of the basis $\Phi_u$ of $\WQSym$. We shall then represent  a cone by
a certain Laurent series and prove that the series corresponding to both sides
of the identity coincide over a nonempty open set.

The minimal example of Theorem \ref{th1Ku} is
\begin{equation}
\M_1 \M_1 = \M_{12} + \M_{21} + \M_{11},
\end{equation}
whose counterpart is
\begin{equation}
\label{eq-k1}
\fc_{K_1\times K_1}=\fc_{K_{12}}+\fc_{K_{21}}-\fc_{K_{11}}
\end{equation}
where
\begin{equation}
\begin{split}
K_1\times K_1 &= (x_1\ge 0, x_2\ge 0), \\
K_{12}        &= (x_1\ge 0, x_1+x_2\ge 0), \\
K_{21}        &= (x_2\ge 0, x_1+x_2\ge 0), \\
K_{11}        &= (x_1+x_2\ge 0). 
\end{split}
\end{equation}

This is to be compared with the product rule \eqref{prodF} of $\FQSym$,
whose minimal example is
\begin{equation}
\F_1\F_1=\F_{12}+\F_{21}\,.
\end{equation}
Although \eqref{prodF} can be derived from the embedding of $\FQSym$ into
$\WQSym$, its geometric interpretation is of a different nature. Indeed, on
the one hand, $\F_\alpha$ can be interpreted as the characteristic
function of a simplex, and the product rule reflects then the classical
decomposition of a product of simplexes.
On the other hand, \eqref{prodKu} is purely a linear relation beween
characteristic functions, and does not follow from a dissection
of the product $K_u\times K_v$, but rather from an argument
of inclusion-exclusion.

\begin{corollary}\label{corQSym}
Let $f$ be a probability distribution over $\RR$, and set, for
$u$ of length $n$
\begin{equation}
m_u(f)=(-1)^{\max(u)}\int_{K_u}f(x_1)\dots f(x_n)dx_1\dots dx_n\,.
\end{equation}
Then $m_u$ depends only on the integer composition $I=\ev(u)$,
so that we can denote it by $m_I$ as well. 
Then the formal series
\begin{equation}
S(f):=\sum_u m_u \N_u
\end{equation}
is grouplike for the coproduct of $\WQSym^*$. If one embeds
$\Sym$ in $\WQSym^*$ by (\ref{SinW}), then, 
\begin{equation}
S(f):=\sum_I m_I S^I
\end{equation}
is grouplike in $\Sym$ (that is, $m_I=\chi(M_I)$ for some character
of $QSym$).
\end{corollary}


\subsection{Changing bases and cones}

To prove Theorem \ref{th1Ku}, it will be easier to work with the fundamental
basis $\Phi_u$ of $\WQSym$.

\begin{proposition}
\label{propCu}
If we identify as above $(-1)^{\max(u)}\M_u$ with the characteristic
function of $K_u$, then $(-1)^{\max(u)}\Phi_u$ gets identified
with the characteristic function of the cone $C_u$, defined
by the conditions
\begin{equation}
\sum_{i=1}^kx_{\sigma_i}\ 
\quad 
\begin{cases}
<0& \ \text{if $\sigma_k$
      is not the end of a block of $\sigma$},\\
\ge 0& \ \text{otherwise.}
\end{cases}
\end{equation}
for $k=1,\dots,n$.
\end{proposition}

\Proof
The change of basis from the $M_u$ to the $\Phi_u$ is given by \eqref{M2Phi}.
It can be rewritten as
\begin{equation}
(-1)^{\max(u)}\M_u = \sum_{v; v\raff u} (-1)^{\max(v)} \Phi_v,
\end{equation}
and the proposition follows from the fact that $K_u$ is the union of the
$C_v$, for $v\raff u$, which is clear from their definitions.
\qed

Now, Theorem~\ref{th1Ku} rewrites in the basis $\Phi_u$ as
\begin{theorem}
\label{th1Cu}
The map $\alpha:\ \PP\rightarrow \WQSym$ defined by
\begin{equation}
\fc_{C_u}\mapsto (-1)^{\max(u)}\Phi_u
\end{equation}
is an isomorphism of algebras.
\end{theorem}

In particular, if $u$ is a nondecreasing packed word, such as $u=111233$ (so
that $\sigma=(123|4|56)$), then the characteristic function of $C_u$ has the
form
\begin{equation} 
\fc_{C_u}
 = \sigma_{\varepsilon_1}(x_1) \sigma_{\varepsilon_2}(x_1+x_2)
   \dots
   \sigma_{\varepsilon_n}(x_1+\dots+x_n)
\end{equation}
so that integrals of $f(x_1)\dots f(x_n)$ over $C_u$ have the form
\eqref{eq-defmf}.
Thus, \eqref{Phi2M} implies that the series $S(f)$ coincides with $R_f$ as
defined in Theorem \ref{thmRm}, so that this is actually a special case
of Theorem~\ref{th1Cu}.

The minimal example of Theorem \ref{th1Cu} is
\begin{equation}
\Phi_1 \Phi_1 = \Phi_{11} + \Phi_{21},
\end{equation}
whose counterpart is
\begin{equation}
\fc_{C_1\times C_1}=\fc_{C_{21}} - \fc_{C_{11}}
\end{equation}
where
\begin{equation}
\begin{split}
C_1\times C_1 &= (x_1\ge 0, x_2\ge 0), \\
C_{21}        &= (x_2\ge 0, x_1+x_2\ge 0), \\
C_{11}        &= (x_1<0, x_1+x_2\ge 0). 
\end{split}
\end{equation}

In particular, since $\Phi_u$ is a lift of $F_I$, we have:
\begin{corollary}\label{corPhi}
If $\chi$ is the character of $QSym$ defined from a function $f$ as in
Corollary \ref{corQSym}, then
\begin{equation}
\chi(F_I)=(-1)^{\max(u)}\int_{C_u}f(x_1)\dots f(x_n)dx_1\dots dx_n
\end{equation}
for any $u$ such that $I=\ev(u)$.
\end{corollary}

\begin{example} (The Sparre Andersen formula) {\rm
Let $f$ be a probability distribution on $\RR$ and $(X_n)_{n\ge 1}$
be a sequence of independent random variables of distribution $f$.
Define $S_n=X_1+\dots+X_n$ and
\begin{equation}
\tau_n = \P(S_1<0,S_2<0,\dots,S_{n-1}<0,S_n\ge 0), \quad \tau(s)=\sum_{n\ge 1}\tau_ns^n,
\end{equation}
and
\begin{equation}
q_n = \P(S_n\ge 0).
\end{equation}
The celebrated formula of E. Sparre Andersen ({\it cf.} \cite[p. 413]{Feller})
states that
\begin{equation}\label{Anders}
\log\frac1{1-\tau(s)} = \sum_{n\ge 1} q_n \frac{s^n}{n}\,.
\end{equation}
This is immediate from Corollary \ref{corPhi}, since
\begin{equation}
\tau_n=-\chi(F_n)\ \text{and}\ q_n=-\chi(M_n). 
\end{equation}
But $F_n$ is the complete homogeneous symmetric function $h_n$
and $M_n$ is the power-sum $p_n$. So \eqref{Anders} follows
by applying $\chi$ to the well-known (Newton) identity
\begin{equation}
\sum_{n\ge 0}h_n s^n =\exp\sum_{n\ge 1}\frac{p_n}{n}s^n\,.
\end{equation}
Another explanation of the Sparre Andersen formula relying on different Hopf
algebras \cite{Man} will appear in the doctoral thesis of A. Mansuy. 
}
\end{example}

\bigskip

Theorems~\ref{th1Ku} and~\ref{th1Cu} being equivalent, we shall prove the
latter.

\subsection{The integer point transform}

Since a polyhedral cone is characterized by the set of its integral points,
we shall encode $C_u$ by the Laurent series
\begin{equation}
F_u := (-1)^{\max(u)}\sum_{\alpha\in C_u\cap\ZZ^n}z_1^{\alpha_1}
z_2^{\alpha_2}\dots z_n^{\alpha_n}\,.
\end{equation}

This identification endows the vector space spanned by the series $F_u$ with the $\star$
product.

Such a series has a nonempty domain of convergence $D_u$ in $\CC^n$, and
inside it, represents a rational function $f_u(x)$. 
The pair $(f_u, D_u)$ allows to reconstruct $C_u$ unambiguously.
It is sometimes called the integer point transform of the polyhedral cone
$C_u$~\cite{BR}.

Note that we could equivalently work with the Laplace transform
of the characteristic function, and identify $C_u$ with
\begin{equation}
(-1)^{\max(u)}\int_{C_u}e^{-\<p,x\>}dx.
\end{equation}
This is again a rational function, which, together with
the domain of convergence of the integral, allows the reconstruction
of $C_u$.

For example,
\begin{equation}
\begin{split}
F_{11} &=
  - \sum_{\alpha_1\leq-1} z_1^{\alpha_1}
    \sum_{\alpha_2\geq-\alpha_1} z_2^{\alpha_2} \\
&=
  - \sum_{\alpha_1\leq-1} z_1^{\alpha_1} z_2^{-\alpha_1}
    \sum_{\alpha'_2\geq 0} z_2^{\alpha'_2}
\end{split}
\end{equation}
so that
\begin{equation}
F_{11} = \left(- \frac{1}{1-z_2} \frac{z_2/z_1}{1-z_2/z_1}\,, 
|z_2|<|z_1|,|z_2|<1\right)\,.
\end{equation}
We also have
\begin{equation}
F_{12} := \left(\frac{1}{1-z_2} \frac{1}{1-z_1/z_2}\,, 
|z_1|<|z_2|<1\right)\,,
\end{equation}
\begin{equation}
F_{21} := \left(\frac{1}{1-z_1} \frac{1}{1-z_2/z_1}\,, 
|z_2|<|z_1|<1\right)\,,
\end{equation}

More generally, we have
\begin{proposition}
\label{pack2rat}
Let $u$ be a packed word, and $\sigma$ be the corresponding
segmented permutation. Then, the rational function associated
with $F_u$ is
\begin{equation}
f_u = (-1)^{\max(u)} \frac{1}{1-z_{\sigma_n}}
\prod_{i=1}^{n-1} g(\sigma_i,\sigma_{i+1}),
\end{equation}
where
\begin{equation}
\label{eq-seg}
g(\sigma_i,\sigma_{i+1}) = \left\{
\begin{array}{rl}
\frac{z_{\sigma_{i+1}}/z_{\sigma_{i}}}{1-z_{\sigma_{i+1}}/z_{\sigma_{i}}}
& \text{if $\sigma_i$ and $\sigma_{i+1}$ are not separated by a bar,}
\\[8pt]
\frac{1}{1-z_{\sigma_{i}}/z_{\sigma_{i+1}}}
& \text{otherwise.} 
\end{array}\right.
\end{equation}
\end{proposition}

\Proof 
Let us write the expansion of $(-1)^{\max(u)} f_\sigma$ as
\begin{equation}
\sum_{i_1 \rR_1 0} z_{\sigma_1}^{i_1}
\sum_{i_1+i_2 \rR_2 0} z_{\sigma_2}^{i_2}
\sum_{i_1+i_2+i_3 \rR_3 0} z_{\sigma_3}^{i_3} \dots
\end{equation}
where $\rR_i$ is either $\geq$ or $<$ depending on whether there is a bar after
$i$ in $\sigma$ or not.

Let $i'_2=i_1+i_2$. Then,
\begin{equation}
\phi_\sigma =
\sum_{i_1 \rR_1 0} z_{\sigma_1}^{i_1} z_{\sigma_2}^{-i_1}
\sum_{i'_2 \rR_2 0} z_{\sigma_2}^{i'_2} 
\sum_{i'_2+i_3 \rR_3 0} z_{\sigma_3}^{i_3} \dots
\end{equation}
This sum splits therefore into two partial sums, the first one being equal to
$g(\sigma_1,\sigma_2)$, the remaining part being $\pm f_{(\sigma_2,\dots)}$.
\qed

This expression of $f_u$ can be simplified: indeed, both choices
of~\eqref{eq-seg} are equal up to sign when one does not consider the domain
of convergence. So, if one writes the denominators in the form
$z_{\sigma_i}-z_{\sigma_{i+1}}$, the terms with a $-1$ match the bars, so that
we have

\begin{corollary}
The function $f_u=f_\sigma$ simplifies as
\begin{equation}
f_\sigma = \frac{1}{z_{\sigma_n}-1}
  \prod_{i=1}^{n-1} \frac{z_{\sigma_{i+1}}}{z_{\sigma_i} - z_{\sigma_{i+1}}}.
\end{equation}
\end{corollary}

\bigskip
For example,
\begin{equation}
\phi_{211} = \phi_{(23|1)}
= + \frac{1}{1-z_1}
    \frac{1}{1-z_3/z_1} \frac{z_3/z_2}{1-z_3/z_2}
= \frac{z_3z_1}{(z_2-z_3)(z_3-z_1)(z_1-1)}.
\end{equation}
\begin{equation}
\begin{split}
\phi_{1223} = \phi_{(1|23|4)}
&= - \frac{1}{1-z_4} \frac{1}{1-z_3/z_4}
    \frac{z_3/z_2}{1-z_3/z_2} \frac{1}{1-z_1/z_2} \\
&
= \frac{z_2z_3z_4}{(z_1-z_2)(z_2-z_3)(z_3-z_4)(z_4-1)}.
\end{split}
\end{equation}

\begin{theorem}
\label{thm-starLaur}
The $\star$ product of the Laurent series $F_\sigma$ is given by
\begin{equation}\label{eqprodseries}
F_{\sigma}\star F_{\sigma'}
 = \sum_{\sigma''\in \sigma\ssh\sigma'} F_{\sigma''}.
\end{equation}
that is, by the same formula as the $\Phi_\sigma$ of $\WQSym$.
\end{theorem}

\Proof
Consider the product $F_{\sigma}\star F_{\sigma'}$ with $\sigma$ of length $n$
and $\sigma'$ of length $p$.
The domain of convergence of this product is a nonempty open set $\O$, which
is the Cartesian product of an open set of $\CC^n$ and an open set of $\CC^p$.

Then, the intersection of $\O$ with the domains $|z_j|<|z_i|<1$ for all
$j\in[n+1,n+p]$ and all $i\in[1,n]$ is again a nonempty open set $\O'$.

All the $F_{\sigma''}$ occurring in the r.h.s. of~\eqref{eqprodseries}
converge in $\O'$, by definition of the segmented shifted shuffle,
so that both sides of~\eqref{eqprodseries} define rational functions in $\O'$.
Thus, we only have to prove equality of these rational functions.

Consider the right-hand side $h$ as a function of $z_{\sigma_1}$.
Since for each fraction, the numerator has a degree strictly lower
than the denominator, this is also true for their sum. Now consider $h$ as a
reduced fraction. Apart from $z_{\sigma_2}$, the possible poles of $z_{\sigma_1}$
are the $z_{\sigma'_j}$ for any $j$. But these poles do not occur: consider
the residue. Multiplying $h$ by $(z_{\sigma_1}-z_{\sigma'_j})$ and then
putting $(z_{\sigma_1}=z_{\sigma'_j})$ yields zero in all permutations where
$\sigma_1$ and $\sigma'_j$ are not neighbours, and the remaining permutations
can be paired as $\alpha \sigma_1\sigma'_j\beta$,
$\alpha \sigma'_j|\sigma_1\beta$ whose common residue is also zero.
So as a reduced fraction in $z_{\sigma_1}$, $h$ has a denominator of degree
$1$ and a numerator of degree $0$, hence is equal to a constant divided by
$z_{\sigma_1}-z_{\sigma_2}$.

Now, putting $z_{\sigma_1}=0$, the term corresponding to a segmented
permutation $\tau$ either gives zero if $z_{\tau_1}\not= z_{\sigma_1}$
or gives $-f_{\tau_2\dots}$ otherwise. By induction, this gives the desired
formula.
\qed

Let us illustrate the proof on the example of the product $F_{11} F_{21}$.
We have
\begin{equation}
\begin{split}
f_{11}\star f_{21} = f_{(12)} \star f_{(2|1)}
 = \frac{z_2}{(z_1-z_2)(z_2-1)}
   \frac{z_3}{(z_4-z_3)(z_3-1)}
\end{split}
\end{equation}
and the sum $f_{1121}+f_{1221}+f_{1321}+f_{2221}+f_{2321}+f_{3321}$ which is
also $f_{(124|3)}+f_{(14|23)}+f_{(14|3|2)}+f_{(4|123)}+f_{(4|13|2)}
+f_{(4|3|12)}$ in terms of segmented permutations translates into  fractions as
\begin{equation}
\begin{split}
& \quad
  \frac{z_2z_4z_3}{(z_1-z_2)(z_2-z_4)(z_4-z_3)(z_3-1)}
+ \frac{z_4z_2z_3}{(z_1-z_4)(z_4-z_2)(z_2-z_3)(z_3-1)}
\\ &
+ \frac{z_4z_3z_2}{(z_1-z_4)(z_4-z_3)(z_3-z_2)(z_2-1)}
+ \frac{z_1z_2z_3}{(z_4-z_1)(z_1-z_2)(z_2-z_3)(z_3-1)}
\\ &
+ \frac{z_1z_3z_2}{(z_4-z_1)(z_1-z_3)(z_3-z_2)(z_2-1)}
+ \frac{z_3z_1z_2}{(z_4-z_3)(z_3-z_1)(z_1-z_2)(z_2-1)}.
\end{split}
\end{equation}
Now, the permutations having a pole $z_1-z_2$
are $(124|3)$, $(4|123)$, and $(4|3|12)$, that is, if one forgets their $1$,
$(24|3)$, $(4|23)$, and $(4|3|2)$, which are the permutations up to rescaling
belonging to $(1)\ssh (2|1)$.
As for the poles of the r.h.s, let us consider for example
the pole $z_2-z_4$. In that case, it appears in the words $(124|3)$ and
$(14|23)$. Then one easily checks that starting with
$f_{(124|3)}+f_{(14|23)}$, multipliying by $z_2-z_4$ and then
putting $z_2=z_4$, one finds $0$.
The same holds for the pole $z_1-z_4$, where one regroups
$(14|23)$ with $(4|123)$ and $(14|3|2)$ with $(4|13|2)$.

\begin{note}{\rm
Theorem~\ref{thm-starLaur} proves in particular that the $m_u(f)$
satisfy the same product formula as the $\M_u$ of $\WQSym$, hence that the
$m_I$ are the images of the $M_I$ by a character of $QSym$, which concludes
the proof of Theorem~\ref{th1Cu} and of its equivalent form
Theorem~\ref{th1Ku}.
}
\end{note}

Another  proof, exploiting the natural  splitting of the shuffle structure,  is provided in
Section~\ref{sec:newcones} in a broader context: set partitions 
each part are replaced by multisets paritions, allowing to consider integrals on a more
general class of cones.

\section{Rational moulds for $\WQSym$}\label{ratMould}

\subsection{Set compositions as rational functions}

There is another way to encode elements of $\WQSym$ by rational functions,
which in turn define nonlinear operators on certain function spaces.
Evaluating these operators on a fixed function gives then rise to a character
of $\WQSym$, which factors through $QSym$ when the target space is a
commutative algebra.

Let $z_i$, $i\ge 1$ be indeterminates. 
For a set of integers $I$, let $z^I=\prod_{i\in I}z_i$, and for a packed word
$u$ encoding a set composition $\Pi(u)=(B_1,\dots,B_m)$, define the rational
function
\begin{equation} 
M_u(Z)
 = \prod_{k=1}^m
         \left( \prod_{i=1}^k z^{B_i} -1 \right)^{-1}.
\end{equation}
For example, with $u=2131231$, $\Pi(u)=(\{2,4,7\},\{1,5\},\{3,6\})$,
and
\begin{equation} 
M_{2131231}(Z)=\frac1{(z_2z_4z_7-1)(z_2z_4z_7z_1z_5-1)
(z_2z_4z_7z_1z_5z_3z_6-1) }.
\end{equation}
We endow the algebra of rational functions $\CC(Z)$ with the shifted product
\begin{equation} 
f(z_1,\dots,z_p)\star g(z_1,\dots,z_q)
=
f(z_1,\dots,z_p)\, g(z_{p+1},\dots,z_{p+q}).
\end{equation}
The resulting structure is called the rational mould algebra \cite{Cha,CHNT}.

\begin{lemma}
\label{lemphi}
The linear map $\phi:\ \WQSym\rightarrow \CC(Z)$ defined by
$\phi(\M_u)=M_u(Z)$ is an injective homomorphism of algebras for
the $\star$ product on $\CC(Z)$.
\end{lemma}

\Proof The Laurent expansion of $M_u(Z)$ in the domain ($|z_k|>1$ for all $k$)
is
\begin{equation} 
M_u(Z) = \sum_{\alpha\in\ZZ_{<0}^n;\ \pack(\alpha)=u}z^\alpha
\end{equation}
where the packing of words over the negative integers
is defined w.r.t. the natural order ({\it e.g.},
$\pack(-3,-5,-3,-8)=3231$). 
Moreover, the expansions of different $M_u$ have no monomial in
common, so that they are linearly independent.
\qed

This embedding is the perfect analog of the embedding of $\FQSym$ in the
rational mould algebra defined in \cite{CHNT}.
In this reference, rational functions encode operators on formal integrals.
There is an analogous situation here, which can be first understood in terms
of the ordinary Fourier transform.

\subsection{Associated operators}

Let us consider functions $f$ from the real line to some associative algebra,
represented in terms of their Fourier transforms as
\begin{equation}
f(t)=\int_{-\infty}^\infty \hat f(\nu)e^{2i\pi \nu t}d\nu\,.
\end{equation}
Introduce new variables $\nu_k$
and set $z_k=e^{2i\pi \nu_k}$. Assuming
convergence of the integrals, we can now interpret
the mould $M_u(z_1,\dots, z_m)$ as the $m$-linear operator
\begin{equation}
f_1\otimes\dots\otimes f_m \mapsto g=M_u[f_1,\dots,f_m]
\end{equation}
where
\begin{equation}
g(t)=\int_{\RR^m}\hat f_1(\nu_1)
\dots\hat f_m(\nu_m)M_u(e^{2i\pi \nu_1},\dots,
e^{2i\pi \nu_m})
e^{2i\pi \nu_1t}\dots
e^{2i\pi \nu_m t}d\nu_1\dots d\nu_m.
\end{equation}
Expanding $M_u$ as a Laurent series near infinity  and
regarding it as  the Fourier series of a distribution,
we see that the result is
\begin{equation}
g(t) = \sum_{\alpha\in\ZZ_{<0}^n;\ \pack(\alpha)=u}
f_1(t+\alpha_1)f_2(t+\alpha_2)\dots f_m(t+\alpha_m)\,.
\end{equation}
For $f_1=f_2=\dots=f_m=f$, we set simply $g(t)=M_u[f](t)$.
For example,
\begin{equation}
M_1[f](t)=\sum_{n\ge 1}f(t-n)
\end{equation}
\begin{equation}
M_{11}[f_1,f_2](t)=\sum_{n\ge 1}f_1(t-n)f_2(t-n)
\end{equation}

\begin{equation}
M_{12}[f_1,f_2](t)=\sum_{n_1,n_2\ge 1}f_1(t-n_1-n_2)f_2(t-n_2)
=M_{21}[f_2,f_1](t).
\end{equation}

We can now rephrase Lemma \ref{lemphi} as follows:
\begin{proposition}
If we set, for $|u|=p$ and $|v|=q$, 
\begin{equation}
(M_u\star M_v) [f_1,\dots f_{p+q}]
 = M_u[f_1\dots,f_p]\, M_v[f_{p+1},\dots,f_{p+q}]
\end{equation}
then, the map $\M_u\rightarrow M_u$ is an injective homomorphism.
\end{proposition}

Observing that 
\begin{equation}
M_{11}[f_1,f_2]=M_1[f_1f_2]\ \text{and}\
M_{12}[f_1,f_2]=M_1[M_1[f_1]f_2]
\end{equation}
and that
\begin{equation}
\M_1^2=
\M_{12}+\M_{21}+ \M_{11},
\end{equation}
we have
\begin{equation}\label{RBM}
M_1[ f_1M_1[f_2] + M_1[f_1]f_2 + f_1f_2] = M_1[f_1]M_1[f_2]\,.
\end{equation}
Let us set for short $M=M_1$. Then, (\ref{RBM}) means that $M$ is
a Rota-Baxter operator (see Section \ref{secRB}), and moreover, we have:

\begin{proposition}
\label{decM}
Any operator $M_u$ can be written as a composition of operators $M$ and
products of functions. More precisely, if $\Pi(u)=(B_1,\dots,B_m)$
is the associated set composition, set 
\begin{equation}
b_i=\prod_{j\in B_i}f_j \quad \text{and\ $[b_i] = M[b_i]$}\,.
\end{equation}
Then,
\begin{equation}\label{eqdecM}
\M_u[f_1,\dots,f_n]
= [\dots [[b_1]b_2]\dots b_m]\,.
\end{equation}
\end{proposition}

\Proof If $v$ is a packed word of length $n$
and $v=u\sigma$ for a permutation $\sigma\in\SG_n$, then
\begin{equation}
M_v[f_1,\dots,f_n]=M_u[f_{\sigma(1)},\dots f_{\sigma(n)}],
\end{equation}
so that it is sufficient to prove the property for nondecreasing
packed words. Now,
\begin{equation}
M_{1^k}[f_1,\dots,f_k]=M[f_1f_2\dots f_k],
\end{equation}
and if $m$ is the maximum letter of a packed word $u$ of length $n$,
then
\begin{equation}\label{addmax}
M[ M_u[f_1,\dots,f_n]f_{n+1}\dots f_{n+k}]=
M_{u(m+1)^k}[f_1,\dots,f_{n+k}]\,.
\end{equation}
\qed

For example,
\begin{eqnarray}
M_{21}[f_1,f_2] &=&   M[M[f_2]f_1]\,,\\
M_{132}[f_1,f_2,f_3] &=& M[f_2 M [f_3 M[f_1]]]\,,\\
M_{3121}[f_1,f_2,f_3,f_4] &=& M[f_1 M[f_3 M[f_2f_4]]]\,.
\end{eqnarray}

\subsection{Operadic considerations}

\subsubsection{Partial compositions}

The embedding of $\FQSym$ in the operad of rational moulds
allowed the identification of various suboperads \cite{CHNT},
in particular the dendriform operad. We shall see that the
embedding of $\WQSym$ yields similar results for the tridendriform operad.

To define an operad structure, we need partial compositions $\circ_k$.
Their definition is transparent on the operators $M_u$. Let $\Delta$ be the
finite difference operator
\begin{equation}
\Delta f(t) =f(t+1)-f(t)\,.
\end{equation}
Note that $\Delta$ is a left inverse for $M$:
\begin{equation}
\Delta M[f](t)=f(t)\,.
\end{equation}
Then, the $k$th partial composition $M_u\circ_k M_v$, with $u$ of length
$m$ and $v$ of length $n$, is defined by
\begin{equation}\label{compM}
\begin{split}
M_u\circ_k M_v[f_1,\dots,f_{m+n-1}]
=M_u[f_1,\dots,f_{k-1},\Delta M_v[f_k,\dots,f_{k+n-1}],\\
f_{k+n},\dots,f_{m+n-1}]\,.
\end{split}
\end{equation}
In terms of the associated rational functions, this reads
\begin{equation}\label{compRat}
\begin{split}
M_u\circ_k M_v(Z)
=&
 (z_kz_{k+1}\dots z_{k+n-1}-1)\\
& \times M_u(z_1,\dots,z_{k-1},z_kz_{k+1}\dots z_{k+n-1},z_{k+n},\dots z_{m+n-1})\\
& \times M_v(z_k,\dots,z_{k+n-1})\,.
\end{split}
\end{equation}
Indeed,
\begin{equation}
\begin{split}
\Delta M_v[f_k,\dots,f_{k+n-1}](t)
=&\int_{\RR^n}\hat f_k(\nu_k)\dots\hat f_{k+n-1}(\nu_{k+n-1})
      (e^{2\pi i(\nu_k+\dots+\nu_{k+n-1})}-1)\\
&\quad\quad \times e^{2\pi i\nu_k t}\dots e^{2\pi i\nu_{k+n-1} t}
d\nu_k\dots d\nu_{k+n-1}
\end{split}
\end{equation}
so that its Fourier transform is
\begin{equation}
\begin{split}
(e^{2\pi i\nu}-1)\int_{\RR^n}
\delta(\nu-\nu_k-\dots-\nu_{k+n-1})
\hat f_k(\nu_k)\dots\hat f_{k+n-1}(\nu_{k+n-1})\\
\quad\quad \times e^{2\pi i\nu_k t}\dots e^{2\pi i\nu_{k+n-1} t}
d\nu_k\dots d\nu_{k+n-1}\,.
\end{split}
\end{equation}
Plugging this expression into (\ref{compM}), we obtain (\ref{compRat}).

On this version, it is clear that $M_u$ and $M_v$ can be replaced by arbitrary
rational functions, and it is easy to check that the axioms of an operad are
satisfied. This is an analogue of the operad $\Mould$ of \cite{Cha,CHNT} which
will be denoted here by $\Mould^0$.

\begin{theorem}
(i) The space of rational functions in $Z$ endowed with the partial
compositions $\circ_k$ above acquires the structure of an operad, which will
be denoted by $\Mould^1$. \\
(ii) It is isomorphic to the operad $\lambda$-${\bf RatFct}$ defined by
Loday~\cite{Lod10}, for $\lambda=1$.\\
(iii) The $M_u$ span a suboperad of $\Mould^1$.
\end{theorem}

\Proof 
(i) 
This can be checked directly. However, it follows from (ii).

\noindent (ii) Loday's rule for the composition in  $1$-${\bf RatFct}$
is
\begin{equation}\label{compLod}
\begin{split}
P\circ'_k Q(x_1,\dots x_{m+n-1})
= &P(x_1,\dots,x_{k-1},\theta^1(x_k,\dots,x_{k+n-1}),x_{k+n},\dots,x_{m+n-1})\\
&\quad\quad \times Q(x_k,\dots,x_{k+n-1})\,
\end{split}
\end{equation}
with 
\begin{equation}
\theta^1(x_1,\dots, x_n) = \prod_{i=1}^n(x_i+1) -1\,.
\end{equation}
Set
\begin{equation}
P'(x_1,\dots,x_n) := \left(\prod_{i=1}^nn(x_i+1) -1\right) P(x_1+1,\dots,x_n+1).
\end{equation}
Then
\begin{equation}
P\circ_i Q = \sum_R R \Longleftrightarrow P'\circ'_i Q' = \sum_R R'.
\end{equation}

%

\noindent (iii) By Proposition \ref{decM}, 
\begin{equation}
M_v[f_k,\dots,f_{k+n-1}] = M[F]
\end{equation}
where $F$ is the  product of a term $M_{v'}[f_{i_1},\dots]$ with
some $f_j$. Precisely, $v'$ is obtained from $v$ by erasing its
maximal letter $m$, the $f_j$ are the elements of the last block
$B_m$ of $\Pi(v)$, and the arguments of $M_{v'}$ are the remaining
$f_i$, in their natural order. Hence,
\begin{equation}
\Delta M_v[f_k,\dots,f_{k+n-1}] = F =
M_{v'}[f_{i_1},\dots,f_{i_r}]f_{j_1}\dots f_{j_s}\,.
\end{equation}
Plugging this expression into the decomposition \eqref{eqdecM} of $M_u$,
and applying \eqref{addmax},
we obtain $M_u\circ_k M_v$ as a multiplicity-free sum of terms $M_w$.
\qed

For example,
\begin{equation}
\begin{split}
M_{12}\circ_2 M_{12}[f_1,f_2,f_3] &= M_{12}[f_1,\Delta M_{12}[f_2,f_3]]
                     = M[M[f_1] M[f_2] f_3]\\
&\quad\text{(since $M_{12}[f_2,f_3]=M[M[f_2]f_3]$)}\\
& = M[ M_{12}[f_1,f_2]f_3 +  M_{21}[f_1,f_2]f_3 + M_{11}[f_1,f_2]f_3]\\
& = (M_{123}+M_{213}+M_{112})[f_1,f_2,f_3]\,.
\end{split}
\end{equation}

Similarly, one can check that
\begin{eqnarray}
M_{121}\circ_1 M_{12} &=& M_{1232},\\
M_{121} \circ_2 M_{12}&=& M_{1121}+ M_{1231}+ M_{2132},\\
M_{121}\circ_3 M_{12} &=& M_{2312},
\end{eqnarray}
and
\begin{eqnarray}
M_{123} \circ_1 M_{12} &=& M_{1234},   \\
M_{123} \circ_2 M_{12} &=& M_{1123}+ M_{1234}+ M_{2134},\\
M_{123} \circ_3 M_{12} &=& M_{1213}+ M_{1223}+ M_{1234}+ M_{1324}+M_{2314}\,.
\end{eqnarray}

\subsubsection{Suboperads of $\Mould^1$}

A \emph{dendriform trialgebra} \cite{LRtri}
(or tridendriform algebra)  is an associative algebra whose multiplication
$\odot$ splits into three pieces
\begin{equation}
x\odot y = x\gautrid y + x\miltrid y + x\droittrid y\,,
\end{equation}
where $\miltrid$ is associative, and
\begin{eqnarray}
(x\gautrid y)\gautrid z = x\gautrid (y\odot z)\,,\\
(x\droittrid y)\gautrid z = x\droittrid (y\gautrid z)\,,\\
(x\odot y)\droittrid z = x\droittrid (y\droittrid z)\,,\\
(x\droittrid y)\miltrid z = x\droittrid (y\miltrid z)\,,\\
(x\gautrid y)\miltrid z = x\miltrid (y\droittrid z)\,,\\
(x\miltrid y)\gautrid z = x\miltrid (y\gautrid z)\,.
\end{eqnarray}
The free dendriform trialgebra on one genarator is known
to be based on 
\emph{reduced plane trees}, {\it i.e.},  plane rooted trees in which
each vertex which is not a leaf has at least two children.
These trees are counted by the little Schr\"oder numbers
\cite{LRtri}. 
This algebra is naturally embedded in $\WQSym$.
Indeed, $\WQSym$ is tridendriform,
the partial products being given by
\begin{equation}
\M_{w'} \gautrid \M_{w''} =
\sum_{w=u.v\in w'\convW w'', |u|=|w'| ; \max(v)<\max(u)}
\M_\park,
\end{equation}
\begin{equation}
\M_{w'} \miltrid \M_{w''} =
\sum_{w=u.v\in w'\convW w'', |u|=|w'| ; \max(v)=\max(u)}
\M_\park,
\end{equation}
\begin{equation}
\M_{w'} \droittrid \M_{w''} =
\sum_{w=u.v\in w'\convW w'', |u|=|w'| ; \max(v)>\max(u)}
\M_\park,
\end{equation}

\begin{lemma}\label{optrid}
On the operators $M_u$, these operations translate as
\begin{eqnarray}
(M_u\droittrid M_v)[f_1,\dots,f_{n+m}] &=& 
M[M_u[f_1,\dots,f_n]\Delta M_v[f_{n+1},\dots,f_{m+n}]],\\
(M_u\gautrid M_v)[f_1,\dots,f_{n+m}] &=& 
M[\Delta M_u[f_1,\dots,f_n]M_v[f_{n+1},\dots,f_{m+n}]],\\
(M_u\miltrid M_v)[f_1,\dots,f_{n+m}] &=& 
M[\Delta M_u[f_1,\dots,f_n]\Delta M_v[f_{n+1},\dots,f_{m+n}]],
\end{eqnarray}
\end{lemma}

\Proof Direct verification. \qed

With each reduced plane tree $T$, one can associate an element  $\MM_T$ of
$\WQSym$, defined by means of a map $\TT$ from words to trees \cite{NTtri}.
From each word
$w$ of length $n$, we build a plane tree $\TT(w)$ recursively defined
as follows. If $m=\max(w)$ and $w$ has exactly $k$ occurrences of $m$,
$\TT(w)$ is obtained from the factorization 
\begin{equation}\label{fact}
w=v_0\,m\,v_1\,m\,v_2\dots v_{k-1}\,m\,v_k
\end{equation}
by grafting
$\TT(v_0),\TT(v_1),\dots,\TT(v_k)$ (in this order) on a common root.
We then set
\begin{equation}\label{MT}
\MM_T = \sum_{\TT(w)=T}w
      = \sum_{\TT(u)=T}\M_u  \,.
\end{equation}
One can show \cite{NTtri} that these polynomials span a Hopf subalgebra of
$\WQSym$,
which is precisely the free algebra on one generator for the tridendriform operad.
This generalizes the embedding in $\FQSym$ of the free  algebra on one
generator for the dendriform operad, which is based on planar binary
trees~\cite{LR1,HNT-pbt}.

As in the case of binary trees and $\FQSym$, the sum of the $M_u$ for all $u$
having a given Schr\"oder tree $\TT(u)=T$ is extremely simple:

\begin{theorem}
The mould associated with a plane tree is
\begin{equation}
\MM_T(Z)=\sum_{\TT(u)=T}M_u(Z)=
\prod_{\bullet\in T}\frac1{H_\bullet(Z)}
\end{equation}
where $\bullet$ runs over the internal nodes of $T$,
and 
\begin{equation}
H_\bullet(z)=\left(\prod_{z\in V(T_\bullet)}z-1\right)
\end{equation}
where $T_\bullet$ is the subtree with root $\bullet$ and $V(T_\bullet)$ the
set of the variables labeling its sectors.
\end{theorem}

\Proof 
In terms of the tridendriform operations, $\MM_T$ is given in~\cite{NTpark}
as
\begin{equation}
\MM_T = (\MM_{T_1}\droittrid\M_{1}) \miltrid (\MM_{T_2}\droittrid\M_{1})
\miltrid \dots \miltrid (\MM_{T_{k-1}}\droittrid\M_{1}) \gautrid \M_{T_k}
\end{equation}
if $T$
has as subtrees of its root $T_1,\dots,T_k$ in this order.
The result follows then from Lemma \ref{optrid}.
\qed

For example, the mould associated with the tree 
\begin{equation}
\vcenter{\xymatrix@C=0.5mm@R=4mm{
*{} & *{} & *{} & *{} & {}\ar@{-}[dlll]\ar@{-}[d]\ar@{-}[drrrr] \\
*{} &  {}\ar@{-}[dl]\ar@{-}[dr] & *{}
& *{z_2} & {}\ar@{-}[dl]\ar@{-}[dr] & *{}
& *{z_4} & *{} & {}\ar@{-}[dll]\ar@{-}[d]\ar@{-}[drr] \\
 {}
& *{z_1} &  {} &  {}
& *{z_3} &  {} &  {}
& *{z_5} &  {}
& *{z_6} &  {} \\
}}
\end{equation}
is
\begin{equation}\label{moultree}
\frac1{(z_1-1)(z_3-1)(z_5z_6-1)(z_1z_2z_3z_4z_5z_6-1) }.
\end{equation}
Indeed, the tridendriform expression for $\MM_T$ is
\begin{equation}
(M_1\droittrid M_1) \miltrid (M_1\droittrid M_1)\gautrid M_1
\end{equation}
and the associated operator is therefore
\begin{equation}
\MM_T[f_1,f_2,f_3,f_4,f_5,f_6]=
M\left[
M[f_1]f_2 M[f_3] f_4 M[f_5f_6]
\right]
\end{equation}
whose rational mould is indeed given by \eqref{moultree}.

\subsection{Characters}

Evaluating the $M_u$ on a fixed function $f$ yields a character of $\WQSym$.
For example, the natural character of $\WQSym$ \cite{HNT}
\begin{equation}
\chi(\M_u) = \binom{t}{\max(u)}
\end{equation}
is obtained by choosing for $f$ the characteristic function
of $[0,\infty)$:
\begin{equation}
M_u[\fc_{\R_+}](t) = \binom{t}{\max(u)}.
\end{equation}
Another more trivial character is obtained by choosing
$f(t)=q^t$. The result is simply $q^{(n+1)t}M_u(z_1=q,\dots,z_m=q)$.
Combining both gives
\begin{equation}
\M_u \mapsto M_I\left(\frac{1-q^t}{1-q} \right),
\end{equation}
the specialization of the monomial quasi-symmetric function
to the $q$-integer alphabet $[t]_q$, 
(when $t$ is an integer) as defined in \cite[Eq. (169)]{NCSF2}.

\section{Matrix quasi-symmetric functions ($\MQSym$)}
\label{sec:MQSym}

\subsection{Multiwords and packed matrices}

Almost all known combinatorial Hopf algebras arise as quotients or subalgebras
of an algebra based on \emph{packed integer matrices}, {\it i.e.}, matrices of
nonnegative integers (of arbitrary size) without null rows or columns.
This is $\MQSym$, the Hopf algebra of matrix quasi-symmetric functions,
introduced in \cite{NCSF6}. We shall present here only
the required background and refer the reader
to~\cite{NCSF6} for more details.

Matrices of nonnegative integers are in bijection with ordered partitions
of multisets, that is, sequences of sets which may contain several occurrences
of the same element. Indeed, given a matrix $M$, the coefficient $M_{ij}$ is
the number of occurrences of the letter $j$ in the $i$-th part $P_i$ of an
ordered partition $P$.
The condition that no row is empty translates as no multiset of the sequence
is empty, and the condition that no column is empty that the union of all
multisets is packed (compare with the definition of packed words).
For example,
\begin{equation}
\left(\begin{array}{cccc}
       2 & 1 & 0 & 0\\
       1 & 0 & 0 & 0\\
       0 & 0 & 3 & 1
\end{array}\right) 
\longleftrightarrow
\{1, 1, 2\}\{1\}\{3, 3, 3, 4\}\,.
\end{equation}

Let $\tmmathbf{A} =\tmmathbf{A}' \cdot A$
and $\tmmathbf{B} =\tmmathbf{B}' \cdot B$ be two multiset partitions
with last parts $A$ and $B$.
In order to define their product, as in the case of $\WQSym$, we shall first
shift the values in $\tmmathbf{B}$ by the maximum value $m$
of $\tmmathbf{A}$.
Let us denote this shifted set by $\tmmathbf{B}[m]$.
Then their product
$\pi(\tmmathbf{A},\tmmathbf{B})=\rho(\tmmathbf{A},\tmmathbf{B}[m])$
is defined by
\begin{equation} 
\begin{split}
\rho (\tmmathbf{A}, \tmmathbf{B})
& =  \rho (\tmmathbf{A}, \tmmathbf{B}') \cdot B
   + \rho (\tmmathbf{A}', \tmmathbf{B}) \cdot A\\
&\ + \rho (\tmmathbf{A}', \tmmathbf{B}') \cdot (A \cup B).
\end{split} 
\end{equation}

For example,
\begin{equation} 
\begin{split}
\pi(\{1, 1, 2\}\{1\}, \{1, 1, 1, 2\} )
&= \rho(\{1, 1, 2\}\{1\}, \{3, 3, 3, 4\}) \\
&=
\{1, 1, 2\}\{1\}\{3, 3, 3, 4\}+\{1, 1, 2\}\{3, 3, 3, 4\}\{1\}\\
&\ \ +\{3, 3, 3, 4\}\{1, 1, 2\}\{1\} +\{1, 1, 2, 3, 3, 3, 4\}\{1\}\\
&\ \ +\{1, 1, 2\}\{1, 3, 3, 3, 4\}.\\
\end{split}
\end{equation}

This is not the original definition of $\MQSym$ but this version is better
suited for our purpose. Moreover it is easy to see the embedding of $\WQSym$
into $\MQSym$: it is the linear span of the elements ${\tmmathbf{P}}$ such
that the standard partition $\tmmathbf{P}$ is an ordered partition of the
{\tmstrong{set}} (not multiset!) $[\max (\tmmathbf{P})]$. It is obviously a
subalgebra of $\MQSym$.

\section{Polyhedral cones associated with multiset compositions}
\label{sec:newcones}

\subsection{Main Results}

An alphabet $X =\{x_1 < \dots < x_p \}$ can be identified with the tuple of
coordinate functions $(x_1, \dots, x_p)$ of $\RR^p$.
Then, one can associate with a multiset $A$ on $X$ the sum of its elements,
{\it e.g.},
\begin{equation}
s_{\{x_1, x_1, x_3, x_4, x_4, x_4 \}} = 2 x_1 + x_3 + 3 x_4 .
\end{equation}
With this identification, one can associate with an ordered multiset partition
$\tmmathbf{A}(X) = A_1 (X) \dots A_r (X)$ 
($r = l (\tmmathbf{A})$) 
a subset
$K_{\tmmathbf{A}(X)}$ of $\RR^p$ defined by the inequalities 

\begin{equation} 
s_{A_1 (X) \dots A_i (X)} = \sum_{j = 1}^i s_{A_j (X)} = s_{A_1 (X) \cup
   \dots \cup A_i (X)} \geqslant 0 \hspace{1em} \tmop{for} \hspace{1em} i =
   1, \dots, r ,
\end{equation}
whose characteristic function is
\begin{equation} 
\tmmathbf{1}_{\tmmathbf{A}(X)} = \prod_{i = 1}^r \sigma_+ \left( s_{A_1
   (X) \cup \dots \cup A_i (X)} \right) ,
\end{equation}
where $\sigma_+ =\tmmathbf{1}_{\RR^+}$.

If we consider 
\begin{equation}
\mathcal{F}= \bigoplus_{n \geqslant 0} \mathcal{F}(\RR^n, \RR) 
= \bigoplus_{n \geqslant 0} \mathcal{F}_n
\end{equation}
as a graded algebra for the product 
$\pi(f,g)=f\star g$, {\it i.e.}, on
$\mathcal{F}_p \times \mathcal{F}_q$ 
\begin{equation} 
\pi (f, g) (x_1, \dots, x_{p + q}) = f (x_1, \dots, x_p) g (x_{p + 1},
   \dots, x_{p + q}) ,
\end{equation}
we have the following generalization of Theorem \ref{th1Ku}.

\begin{theorem}\label{thMQSym}
The linear map $\alpha$ defined from multiset compositions (of an
alphabet $X =\{x_1 < \dots < x_p \}$) to \ $\mathcal{F}(\RR^p,
\RR)$ by
\begin{equation}
\alpha (\tmmathbf{A}(X)) = (- 1)^{l (\tmmathbf{A})}
\tmmathbf{1}_{\tmmathbf{A}(X)}
\end{equation}
is such that
\begin{equation}\label{prodmul}
 \pi (\alpha (\tmmathbf{A}(X)), \alpha (\tmmathbf{B}(Y)))
 = \alpha (\pi (\tmmathbf{A}(X), \tmmathbf{B}(Y))).
\end{equation}
\end{theorem}

\begin{proof}
The proof proceeds by induction and relies upon the following simple identity
(compare with~\eqref{eq-k1}).
For two real numbers $a$ and $b$,
\begin{equation}
\label{eqsimple} 
\sigma_+ (a) \sigma_+ (b) = \sigma_+ (a) \sigma_+ (a + b) + \sigma_+ (b)
     \sigma_+ (a + b) - \sigma_+ (a + b) .
\end{equation}
If $l(\mathbf{A})=l(\mathbf{B})=1$, then
$\pi (\alpha (\tmmathbf{A}(X)), \alpha (\tmmathbf{B}(Y)))
=\sigma_+(s_{\tmmathbf{A}(X)}) \sigma_+(s_{\tmmathbf{B}(Y)})$
and (\ref{eqsimple}) with $a =s_{\mathbf{A(X)}}$ and $b= s_{\mathbf{B(Y)}}$
gives (\ref{prodmul}) in the case $l(\mathbf{A})=l(\mathbf{B})=1$.
Let us assume now that Equation (\ref{prodmul}) holds for
$l(\mathbf{A})+l(\mathbf{B})\leq n$, for a given $n\geq 2$.
If $\tmmathbf{A}(X) =\tmmathbf{A}' (X) A (X)$ (resp. $\tmmathbf{B}(Y)
  =\tmmathbf{B}' (Y) B (Y)$) with $l(\mathbf{A})+l(\mathbf{B})=n+1$,
then,
\begin{equation}
\alpha (\tmmathbf{A}(X)) = \alpha (\tmmathbf{A}' (X) A (X))\\
= - \alpha (\tmmathbf{A}' (X)) \sigma_+ (s_{\tmmathbf{A}(X)})
\end{equation}
and
\begin{equation}
\alpha (\tmmathbf{B}(Y)) = \alpha (\tmmathbf{B}' (Y) B (Y))\\
= - \alpha (\tmmathbf{B}' (Y)) \sigma_+ (s_{\tmmathbf{B}(Y)}),
\end{equation}
with the convention $\alpha (\tmmathbf{A}' (X))=1$
(resp. $\alpha (\tmmathbf{B}' (Y))=1$)
if $\tmmathbf{A}'$ is empty (resp. $\tmmathbf{B}'$ is empty).
We have 
\begin{equation}
\begin{split}
\pi (\alpha (\tmmathbf{A}(X)), \alpha (\tmmathbf{B}(Y)))
& = \pi (\alpha (\tmmathbf{A}' (X)) \sigma_+ (s_{\tmmathbf{A}(X)}), \alpha
       (\tmmathbf{B}' (Y)) \sigma_+ (s_{\tmmathbf{B}(Y)}))\\
& = \pi (\alpha (\tmmathbf{A}' (X)), \alpha (\tmmathbf{B}' (Y)))
       \sigma_+ (s_{\tmmathbf{A}(X)}) \sigma_+ (s_{\tmmathbf{B}(Y)})
\end{split}
\end{equation}
and using once again Formula~(\ref{eqsimple})
with $a =s_{\tmmathbf{A}(X)}$ and $b= s_{\tmmathbf{B}(Y)}$,
the product $\pi (\alpha (\tmmathbf{A}(X)), \alpha (\tmmathbf{B}(Y)))$ splits
into three terms:
\begin{equation}
\begin{split}
\pi (\alpha (\tmmathbf{A}(X)), \alpha (\tmmathbf{B}(Y)))
& = \pi (\alpha (\tmmathbf{A}' (X)), \alpha (\tmmathbf{B}' (Y)))
    \sigma_+ (s_{\tmmathbf{A}(X)}) \sigma_+ (s_{\tmmathbf{A}(X)} +
    s_{\tmmathbf{B}(Y)})\\
&\ \ + \pi (\alpha (\tmmathbf{A}' (X)), \alpha (\tmmathbf{B}' (Y)))
       \sigma_+ (s_{\tmmathbf{B}(Y)}) \sigma_+ (s_{\tmmathbf{A}(X)} +
       s_{\tmmathbf{B}(Y)})\\
&\ \ - \pi (\alpha (\tmmathbf{A}' (X)), \alpha (\tmmathbf{B}' (Y)))
       \sigma_+ (s_{\tmmathbf{A}(X)} + s_{\tmmathbf{B}(Y)}).
\end{split}
\end{equation}
Thanks to the definition of $\alpha$ and of the product on $\mathcal{F}$,
\begin{equation}
\begin{split}
\pi (\alpha (\tmmathbf{A}(X)), \alpha (\tmmathbf{B}(Y)))
& = - \pi (\alpha (\tmmathbf{A}(X)), \alpha (\tmmathbf{B}' (Y)))
       \sigma_+ (s_{\tmmathbf{A}(X)} + s_{\tmmathbf{B}(Y)})\\
&\ \ - \pi (\alpha (\tmmathbf{A}' (X)), \alpha (\tmmathbf{B}(Y)))
       \sigma_+ (s_{\tmmathbf{A}(X)} + s_{\tmmathbf{B}(y)})\\
&\ \ -\pi (\alpha (\tmmathbf{A}' (X)), \alpha (\tmmathbf{B}' (Y)))
       \sigma_+ (s_{\tmmathbf{A}(X)} + s_{\tmmathbf{B}(Y)}) .
\end{split}
\end{equation}
Recursively, we can use (\ref{prodmul}) for the three terms:
\begin{equation}
\begin{split}
\pi (\alpha (\tmmathbf{A}(X)), \alpha (\tmmathbf{B}(Y)))
& = - \alpha (\pi (\tmmathbf{A}(X), \tmmathbf{B}' (Y))) \sigma_+
       (s_{\tmmathbf{A}(X)} + s_{\tmmathbf{B}(Y)})\\
&\ \ - \alpha (\pi (\tmmathbf{A}' (X), \tmmathbf{B}(Y))) \sigma_+
       (s_{\tmmathbf{A}(X)} + s_{\tmmathbf{B}(Y)})\\
&\ \ - \alpha (\pi (\tmmathbf{A}' (X), \tmmathbf{B}' (Y))) \sigma_+
       (s_{\tmmathbf{A}(X)} + s_{\tmmathbf{B}(Y)}) \\
\end{split}
\end{equation}
and, using once again the definition of $\alpha$ and of the product of ordered
partitions, we find
\begin{equation}
\begin{split}
\pi (\alpha (\tmmathbf{A}(X)), \alpha (\tmmathbf{B}(Y)))
& = \alpha (\pi (\tmmathbf{A}(X), \tmmathbf{B}' (Y)) .B (Y)) +
       \alpha (\pi (\tmmathbf{A}' (X), \tmmathbf{B}(Y)) .A (X))\\
&\ \ + \alpha (\pi (\tmmathbf{A}' (X), \tmmathbf{B}' (Y)) . (A (X)
       \cup B (Y)))\\
& = \alpha (\pi (\tmmathbf{A}(X), \tmmathbf{B}(Y)),
\end{split}
\end{equation}
so that the result follows by induction.
\end{proof}

For a standard partition $\tmmathbf{P}= (P_1, \dots, P_r)$
(with $\max (\tmmathbf{P}) = n$ and $l (\tmmathbf{P}) = r$), one can identify
$m_{\tmmathbf{P}}$ with $\tmmathbf{P}(X)$ where $X$ is ``minimal''
($|X| = \max (\tmmathbf{P})$).
For example, if
\begin{equation}
\tmmathbf{P}=\{1, 3, 3, 3\}\{2, 2, 3\}\{3, 3, 3\}\{1, 3, 3\},
\end{equation}
then $\max (\tmmathbf{P}) = 3$, $X =\{x_1, x_2, x_3 \}$,
and
\begin{equation}
\begin{split}
\tmmathbf{1}_{K_{\tmmathbf{P}}}
= \sigma_+ (x_1\!+\! 3 x_3)\,
  \sigma_+ (x_1 \!+\! 2 x_2 \!+\! 4 x_3)\,
  \sigma_+ (x_1 \!+\! 2 x_2 \!+\! 7 x_3)\,
  \sigma_+ (2 x_1 \!+\! 2 x_2 \!+\! 9 x_3).
\end{split}
\end{equation}

\begin{corollary}
The induced linear map from $\MQSym$ to $\mathcal{F}$ (also denoted by
$\alpha$)
\begin{equation}
\alpha (m_{\tmmathbf{P}}) = (- 1)^{l (\tmmathbf{P})}
\tmmathbf{1}_{K_{\tmmathbf{P}}}
\end{equation}
is an algebra morphism (character).
\end{corollary}

\section{The Rota--Baxter approach}\label{secRB}

\subsection{Convolution and iterated integrals.}\label{secintRB}

The construction of $R_f$, $L_f$ and $D_f$ (see section \ref{moyennes}) relies
upon iterated integrals involving convolution of functions. More precisely,
let $\mathcal{A}$ be the vector space of bounded integrable functions whose
restrictions to $\RR^+$ and $\RR^{-*}$ are continuous. Thanks to the
regularization effect of the convolution

\begin{equation}\label{Conv}
f*g(x)=\int_{\RR}f(y)g(x-y) dy,
\end{equation}
the convolution product of two functions of $\mathcal{A}$ is in $\mathcal{A}$
(and even continuous), so that $\mathcal{A}$ is a non-unital commutative
algebra. The second step to define these iterated integrals is to mix the
convolution product with the operations
($\varepsilon=\pm$)
\begin{equation}
\forall f\in \mathcal{A},\
  (P_{\varepsilon}f)(x)=f(x)\sigma_{\varepsilon}(x).
\end{equation}
On the one hand, using the operators $P_{\varepsilon}$, the functions defined
in Section~\ref{moyennes} for $n\geq 1$
\begin{equation}
\left [\begin{array}{ccc}
\varepsilon_1  &\dots&\varepsilon_n \\
f_1&\dots&f_n
\end{array} \right ] (x)=\left (\left [\begin{array}{ccc}
\varepsilon_1  &\dots&\varepsilon_{n-1} \\
f_1&\dots&f_{n-1}
\end{array} \right ]\ast f_n \right )(x)\sigma_{\varepsilon_n}(x)
\end{equation} 
read
\begin{equation}
\left [\begin{array}{ccc}
\varepsilon_1  &\dots&\varepsilon_n \\
f_1&\dots&f_n
\end{array} \right
]=P_{\varepsilon_n}(f_n*P_{\varepsilon_{n-1}}(f_{n-1} *\dots
*P_{\varepsilon_1}(f_1) \dots)).
\end{equation}
On the other hand, if $\mathcal{A}_\varepsilon=P_\varepsilon(\mathcal{A})$, 
$\mathcal{A}_+$ and $\mathcal{A}_-$ are two subalgebras of $\mathcal{A}$ such
that $\mathcal{A}=\mathcal{A}_+\oplus\mathcal{A}_-$
and using equation (\ref{eqsimple}), we get
\begin{equation}
\forall f,g \in \mathcal{A},\quad
P_+(f*g)+P_+(f)*P_+(g)=P_+(f*P_+(g)+P_+(f)*g).
\end{equation}
In other words (see \cite{EFG} or \cite{manchon}), $(\mathcal{A}, P_+)$ is a
Rota-Baxter algebra of weight $1$ and the properties of our iterated integrals
can be derived from the properties of such Rota-Baxter algebras.

\subsection{A short reminder on Rota-Baxter algebras.}

For a given commutative $\KK$-algebra $A$, let us consider the
tensor algebra $T(A)=\KK {\bf 1}\oplus(\oplus_{n\geq 1}A^{\otimes^n})$, with
the quasishuffle product $\pi$: for $(\mathbf{a},\mathbf{b})\in T(A)^2$ and
$(a,b)\in A^2$,
\begin{equation}
\label{QSTens}
 \pi(\mathbf{a}\otimes a,\mathbf{b}\otimes
b)=\pi(\mathbf{a}\otimes a,\mathbf{b})\otimes b
+\pi(\mathbf{a},\mathbf{b}\otimes b)\otimes a +
\pi(\mathbf{a},\mathbf{b})\otimes (ab).
\end{equation}
This algebra is indeed a Hopf algebra for the deconcatenation coproduct (see
\cite{Hof}), thus we can consider the group of characters from $T(A)$ to the
unitarization $\KK 1 \oplus A$ of $A$.

Assume now that $A=A_+\oplus A_-$ where $A_+$ and $A_-$ are subalgebras of
$A$, then $A$ is a Rota-Baxter algebra: if $R$ is the projection of $A$ on
$A_+$, parallel to $A_-$ then we get the Rota-Baxter identity
\begin{equation}\label{RBrel}
\forall (x,y)\in A^2,\quad R (x y) + R (x) R (y) = R (x R (y) + R (x)
y).
\end{equation}

Rota-Baxter algebras have been deeply studied in the framework of
renormalization in perturbative quantum field theory (see \cite{EFG},
\cite{manchon}) and
\begin{proposition}
 The map from $T(A)$ to $\KK 1 \oplus A$ defined by
  $C({\bf 1})=1$ and
\begin{equation}
C (a_1\otimes \dots \otimes a_s)=(- 1)^s R (R(\dots R(R(a_1)a_2)\dots)a_s) 
\end{equation}
is a {\it character} on $T(A)$.
\end{proposition}
This follows immediately from the Rota-Baxter
relation (\ref{RBrel}) and the recursive definition of the quasishuffle
product (\ref{QSTens}).

Now Theorem \ref{thmRm} can be easily deduced from the following corollary:

\begin{corollary} Let $I$ and $V$ be two maps from $A$ to $\KK$ such that
  $I(ab)=I(a)I(b)$ and $V(R(a)R(b))$=0, then
\begin{itemize}
\item The map $I\circ C$ is a character from $T(A)$ to $\KK$.
\item The map $V\circ C$ is an infinitesimal character from $T(A)$ to $\KK$. 
\end{itemize}
\end{corollary}

This corollary is closely related to our previous iterated integrals since,
for any given $a\in A$, the coefficients
\begin{equation}
M_I(a)=M_{i_1,\dots,i_r}=C(a^{i_1}\otimes \dots \otimes a^{i_r})
\end{equation} 
define a symmetrel mould (with values in $A$) on $\KK\<\NN^*\>$ or,
equivalently, an $A$-valued character on the quasishuffle Hopf algebra
$\KK\<\NN^*\>$ (see section \ref{secmould}).

Using this character, if there exist maps $I$ and $D$ on $A$ with values in
$\KK$ such that $I(ab)=I(a)I(b)$ and $D(R(a)R(b))$=0, then the coefficients
\begin{equation}
M_a^{n_1,..,n_s}=I(C(a^{n_1} \otimes \dots \otimes a^{n_s}))
\end{equation}
define a character of the quasishuffle Hopf algebra $\KK\<\NN^*\>$ and the
coefficients
\begin{equation}
D_a^{n_1,..,n_s}=D(C(a^{n_1} \otimes \dots \otimes a^{n_s}))
\end{equation}
define an infinitesimal character.

\subsection{A proof of Theorem \ref{thmRm}.}

These results provide a proof of Theorem \ref{thmRm} when applied to the
Rota-Baxter algebra $\mathcal{A}$ defined in Section~\ref{secintRB},
with $I:\mathcal{A}\rightarrow \RR$ the integral over $\RR$ and
$D:\mathcal{A}\rightarrow \RR$ the evaluation at $x=0$ ($D(f)=f(0)$).

For example, recall that, for a function $f$ in $A$, the associated grouplike
element $R_f$ can be written 

\begin{equation}
R_f=\sum \<R_f,\Lambda\>_I \Lambda^I
\end{equation}
where, for $I=(i_1,..,i_r)\vDash n$,

\begin{equation}
\begin{array}{rcl}
\<R_f,\Lambda\>_I &=& \displaystyle (-1)^{r+n}\int_{\RR} \left
    [ \begin{array}{ccc}
+ &\dots &+ \\
f^{*i_1} & \dots & f^{*i_r}
\end{array} \right ](x) dx 
\end{array}
\end{equation}
but

\begin{equation}
\<R_f,\Lambda\>_I= (-1)^n I(C(f^{i_1}\otimes \dots \otimes f^{i_r}))
\end{equation}
is a character on the quasishuffle algebra $\KK\<\NN^*\>$,
so $R_f$ is grouplike.
The same holds for the primitive element $D_f$.

\section{The Catalan idempotents}
\label{catalan}

Such iterated integrals are difficult to compute in general but for a specific
family of functions, these integrals can be evaluated in closed form and yield
a new family of primitive elements of $\Sym$,  originally introduced in
\cite{e1} with a different interpretation. 
Up to a normalization, they provide new Lie idempotents whose combinatorial
meaning is still under investigation\footnote{Since the first version of the
present paper was released as a preprint, F. Chapoton \cite{Cha1} has found a
combinatorial interpretation of the coefficients of these idempotents on the
natural basis of the free preLie algebra on one generator.}.

\subsection{The Catalan triangle}

Consider the generating series
\begin{equation}
\ca (a, b, t)
 = \frac{1 - (a + b) t - \sqrt{1 - 2 (a + b) t + (b - a)^2 t^2}}{2 a b t}
 = \sum_{n \geq 1} \ca_n (a, b) t^n .
\end{equation}
The coefficients $\ca_n(a,b)$ are homogeneous and symmetric polynomials
in $a$, $b$ of degree $n-1$:
\begin{equation}
\begin{array}{rcc}
\ca_1(a,b)&=&1 \\
\ca_2(a,b)&=&a+b \\
\ca_3(a,b)&=&a^2+3ab+b^2 \\
\ca_4(a,b)&=&a^3+6 a^2b+6ab^2+b^3 \\
\ca_5(a,b)&=&a^4+10a^3b+20a^2b^2+10ab^3+b^4 
\end{array}
\end{equation}
and we recognize the Catalan triangle of Narayana numbers
$T(n,k)=\displaystyle \frac{1}{k}C^{k-1}_{n-1}C^{k-1}_{n}$
(see \cite{fra}):
\begin{equation}
\forall n\geq 1,\quad \ca_n(a,b)=\sum_{i=0}^{n-1}T(n,i+1)a^ib^{n-1-i}.
\end{equation}

For any sequence of signs
$\tmmathbf{\varepsilon} =(\varepsilon_1,...,\varepsilon_{n})$ ($n\geq 1$),
consider its minimal decomposition into stacks of identical signs 
\begin{equation}
\tmmathbf{\varepsilon}= (\varepsilon_1, \dots, \varepsilon_n) =
     (\eta_1)^{n_1} \dots (\eta_s)^{n_s},
\end{equation}
with $\eta_i \not= \eta_{i + 1}$ and $n_1 + \dots + n_s = n$. Then,

\begin{theorem}
\label{thcat}
For $n\geq 1$, the element of $\Sym_{n+1}$ defined by
\begin{equation}
D_{a,b}^{n+1}=\sum_{
\tmmathbf{\varepsilon} 
= (\eta_1)^{n_1} \dots (\eta_s)^{n_s} }
  \left( \prod_{\gf{\eta_i = +}{i < s}} a \right)
  \left( \prod_{\gf{\eta_i = -}{i < s}} b \right)
  {ca}_{n_1} (a, b) \dots {ca}_{n_s} (a, b)
   R_{\tmmathbf{\varepsilon}\bullet}
\end{equation}
is primitive.
\end{theorem}

For example, using the correspondence with the usual noncommutative
ribbon Schur functions,
\[
\begin{array}{rcl} 
D_{a,b}^2 &=& R_2 -R_{11} \\
D_{a,b}^3 &=& (a+b)R_3 -a R_{21} -b R_{12}+(a+b)R_{111} \\
D_{a,b}^4 &=& (a^2+3ab+b^2) R_4 -a(a+b)R_{31}-
abR_{22}-(a+b)bR_{13} \\
 &&+a(a+b)R_{211}+abR_{121}+(a+b)bR_{112}-(a^2+3ab+b^2)R_{1111}
\end{array}
\]


\subsection{Proof of Theorem \ref{thcat}}

We apply Theorem \ref{thmRm} to
\begin{equation}
f (x) =2 (a  \sigma_+ (x) + b \sigma_- (x))e^{-|x|}
\end{equation}
where  $a$, $b$ are two real numbers.
Note that
\begin{equation} 
\int_{\RR} f (x) d x =2( a  + b)\,.
\end{equation}
To explicit the (Catalan) operators associated with this function, we
essentially need to compute the function 
\begin{equation} 
f^{\varepsilon_1, \dots, \varepsilon_n} = \left [\begin{array}{ccc}
\varepsilon_1 & \dots & \varepsilon_n \\
f & \dots & f
\end{array} \right ] \,.
\end{equation}

\begin{lemma}
We have, for $n\geq 1$,
\begin{equation} 
\begin{array}{ccc}
f^{\overbrace{+ \dots +}^n} (x) & = & P_n (a, b, x) \sigma_+ (x) e^{- |x|}, \\
f^{\overbrace{- \dots -}^n} (x) & = & P_n (b, a, x) \sigma_- (x) e^{ -|x|},
\end{array} 
\end{equation}
where $P_n (a, b, x) \in \RR[a, b, x]$ is of degree $n-1$ in $x$ and
homogeneous of degree $n$ in $a$, $b$. Moreover,
\begin{equation} 
P (a, b, x, t) = \sum_{n \geq 1} P_n (a, b, x) t^n = 2u (t) e^{2 u (t) x},
\end{equation}
with
\begin{equation} 
u (t) = \frac{1 - (b - a) t - \sqrt{1 - 2 (a + b) t + (b - a)^2 t^2}}{2}.
\end{equation}
\end{lemma}

\begin{proof}
Let $g_n (x) = f^{\overbrace{+ \dots +}^n} (x) = P_n(a,b,x)
\sigma_+ (x) e^{-| x|}$ and $P_n(x) = P_n(a,b,x)$. We have

\begin{equation} 
\begin{array}{rcl}
g_{n + 1} (x) & = &
   \displaystyle \sigma_+ (x) \int_{\RR} g_n (y) f (x - y) d y\\
              & = &
   \displaystyle  \sigma_+ (x)
   \left( \int_{\RR^+} g_n (y) f^+ (x - y) d y
   + \int_{\RR^+} g_n (y) f^- (x - y) d y \right)\\
              & = &
   \displaystyle 2 \sigma_+ (x)
   \left( \int_0^x P_n (y) e^{- y} a e^{-(x - y)} d y
   + \int_x^{+ \infty} P_n (y) e^{- y} b e^{+(x - y)} d y \right)\\
              & = &
   \displaystyle  2 \sigma_+ (x) e^{- x}
   \left( \int_0^x a  P_n (y) d y
   + e^{2 x} \int_x^{+ \infty} b  P_n (y) e^{- 2 y} d y
       \right).
     \end{array} 
\end{equation}
This equation defines recursively the polynomials $P_n$ and, for the
generating function ($P_1 = 2 a $), we have

\begin{equation} 
P (x, t) = 2t \left( a +  \int_0^x a  P (y, t) dy +
      e^{2 x} \int_x^{+ \infty} b  P (y, t) e^{- 2y} dy
     \right) .
\end{equation}
If we substitute
\begin{equation} 
P (x, t) =2 u (t) e^{2 u (t) x} \,,
\end{equation}
  we get
\begin{equation}  
2u(t) e^{2u(t)  x} = 2a  t + 2a t (e^{2u (t) x} -
     1) +2 b t  \frac{ 2u (t)}{2 -  2u (t)} e^{2u (t)
      x} \,,
\end{equation}
so that
\begin{equation} 
u = a t + b t \frac{u}{1 - u}
\end{equation}
 which gives the expected generating function.
\end{proof}

\begin{lemma}
  Let
\begin{equation} 
\tmop{ca} (a, b, t) = \frac{1 - (a + b) t - \sqrt{1 - 2 (a + b) t + (b -
     a)^2 t^2}}{2 a b t} = \sum_{n \geq 1} \tmop{ca}_n (a, b) t^n .
\end{equation}
Then, for $n \geq 1$,
  \begin{equation}
 \begin{array}{ccc}
       f^{\overbrace{+ \dots +}^n -} (x) & = & a \tmop{ca}_n (a, b) f^- (x)\\
       f^{\overbrace{- \dots -}^n +} (x) & = & b \tmop{ca}_n (a, b) f^+ (x)
     \end{array} 
\end{equation}
\end{lemma}

\begin{proof}
This is the same kind of computation with generating functions as in the
previous lemma. Let $h_n=f^{\overbrace{+ \dots +}^n -} (x)$. Then
\begin{equation}
\begin{array}{rcl}
H(t,x) &=& \displaystyle \sum_{n\geq 1}h_n(x)t^n \\
 &=& \displaystyle \sigma_-(x) \int_\RR P(a,b,y,t)e^{-|y|}\sigma_+(y)
 (f^+(x-y)+f^-(x-y))d y \\
 &=& \displaystyle \sigma_-(x)\int_{\RR^+}P(a,b,y,t)e^{-y} 2b e^{x-y}
 d y \\
&=& \displaystyle 2b\sigma_-(x)e^x \int_{\RR^+}2u(t)e^{(2u(t)-2)y} dy
\\
&=& \displaystyle f^-(x)\frac{u(t)}{1-u(t)} \\
&=& \displaystyle\frac{1}{bt}(u(t)-at)f^-(x)\\
&=& a \tmop{ca}(a,b,t)f^-(x)
\end{array}
\end{equation}
\end{proof}

Roughly speaking, when there is a change of sign, as in
$f^{\overbrace{+ \dots +}^n -} (x)$, we recover, up to a scalar, the
initial function. Theorem~\ref{thcat} follows easily from the
previous lemma: for $n\geq 1$, the coefficients of $D_{a,b}^{n+1}$ in
the basis $R_{\tmmathbf{\varepsilon}\bullet}$ are given by
\begin{equation}
\frac{1}{2}f^{\varepsilon_1,\dots,\varepsilon_n,+}(0)
\end{equation}
If $\varepsilon_1, \dots, \varepsilon_n$ is decomposed into stacks of
  identical signs
\begin{equation} 
\tmmathbf{\varepsilon}= \varepsilon_1, \dots, \varepsilon_n =
     (\eta_1)^{n_1} \dots (\eta_s)^{n_s} 
\end{equation}
  with $\eta_i \not= \eta_{i + 1}$ and $n_1 + \dots + n_s = n$ and
  $\eta_s=+$, then
\begin{equation}
f^{\varepsilon_1,\dots,\varepsilon_n,+}(0)
 = \left( \prod_{\gf{\eta_i = +}{i<s}} a \right)
   \left( \prod_{\gf{\eta_i = -}{i<s}} b \right)
   \tmop{ca}_{n_1} (a, b) \dots
   \tmop{ca}_{n_{s - 1}} (a, b) P_{n_s+1} (a, b, 0)
\end{equation}
and $P_{n_s+1} (a, b, 0)=2ab\, \ca_{n_s}(a,b)$. 
If $\varepsilon_1,\dots,\varepsilon_n$ is decomposed into stacks of identical
signs,
\begin{equation}
\tmmathbf{\varepsilon}= \varepsilon_1, \dots, \varepsilon_n =
     (\eta_1)^{n_1} \dots (\eta_s)^{n_s},
\end{equation}
with $\eta_i \not= \eta_{i + 1}$ and $n_1 + \dots + n_s = n$ and
$\eta_s=-$, then
\begin{equation}
f^{\varepsilon_1,\dots,\varepsilon_n,+}(0)
 = \left( \prod_{\gf{\eta_i = +}{i<s}} a \right)
   \left( \prod_{\gf{\eta_i = -}{i<s}} b \right)
   \tmop{ca}_{n_1} (a, b) \dots
   \tmop{ca}_{n_{s - 1}} (a, b) \ca_{n_s} (a, b) b. 2a
\end{equation}
which ends the proof of the theorem.

\section{Alien calculus and noncommutative symmetric functions}\label{sec:alien}

We shall conclude this paper with a brief introduction to resummation theory, and
explain how $\Sym$ appears in this context as a Hopf algebra of
analytic continuation operators.  

Alien calculus provides a deep understanding of resummation schemes
that allow to interpret a formal power series as the asymptotic
expansion of a function.

Let $\tilde \varphi (z)\in \RR[[z^{-1}]]$ be a divergent
series of ``natural origin'': for instance, the formal solution of a
local analytic equation or system:
\begin{equation} E( \tilde \varphi ) =0\,.
\end{equation}
The  simplest real resummation scheme for $\tilde \varphi (z)$
goes like this:
\begin{equation}
\begin{array}{ccccc} \widetilde \varphi (z) & - & - & \rightarrow &
\varphi(z) \\ & \searrow & & \nearrow & \\ & & \hat \varphi (\zeta) &
&
\end{array}
\end{equation}
We begin by subjecting $\widetilde \varphi (z)$ to the formal Borel
transform (to obtain $ \hat \varphi (\zeta)$), which turns each monomial
$z^{-\sigma}$ into $\zeta^{\sigma-1} \Gamma(\sigma)$ ($\sigma>0$).
Under some growth condition on the coefficients of $\widetilde \varphi (z)$,
its Borel transform is a germ near $\zeta=0$ and it converges only for small
enough values of $\zeta$. 
If this germ can be analytically continued along $\RR^+$ then, under
some growth condition,  we can carry out a Laplace transform:

\begin{equation} \hat \varphi (\zeta) \longrightarrow \varphi(z) =
\int_{0}^{+\infty} e^{-z \zeta} \hat \varphi (\zeta) d \zeta
\end{equation} 
which converges for $\Re(z)>>0$, is real for
$z$ real, and whose asymptotic expansion is $\widetilde \varphi (z)$.
\bigskip

When it is possible, this procedure for turning a real formal object
$\widetilde \varphi (z)$ into a real geometric one $\varphi (z)$ is the
simplest one and preserves the product of functions. Unfortunately, the
analytic continuation of the germ $\hat \varphi (\zeta)$ often gives rise to
analytic singularities on the real axis, which prevents from carrying out the
Laplace transform. When this is the case, a careful analysis of the
singularities is needed. It is provided by Alien Calculus.

In many instances, the Borel transform of a  formal series
$\widetilde \varphi (z)$ lives in an algebra of functions whose product
reflects the product of formal power series. We will focus here on the
following algebra.

\begin{definition}
Let $\Res_\NN$ be the vector space of functions $\hat \varphi (\zeta)$
such that 
\begin{itemize}
\item $\hat \varphi (\zeta)$ is defined and holomorphic
  at the root of
$\RR^+$, that is, on a domain
\begin{equation}
S=\{0<|\zeta|<\varepsilon, |\arg \zeta |<\theta\}.
\end{equation}

\item $\hat \varphi (\zeta)$ is analytically continuable along any
path that follows $\RR^+$ and dodges each point of 
$\NN^*$ to the left or to the right, but without ever going back.
\item All the determinations of $\hat \varphi (\zeta)$ are locally
integrable on $\RR^+$.
\end{itemize}
\end{definition}

This space is an algebra for the convolution product :
\begin{equation} \hat \varphi_3(\zeta) =(\hat \varphi_1 \ast \hat
\varphi_2)(\zeta)=\int_0^{\zeta} \hat \varphi_1(\zeta_1) \, \hat
\varphi_2(\zeta - \zeta_1) \, d \zeta_1 \; \; (0<\zeta < 1)
\end{equation}
where $\hat \varphi_1,\hat \varphi_2 \in \Res_\NN$.

Note that this expression is purely local (at $\zeta = 0$) so that the germ
$\hat \varphi_3(\zeta)$ must then be extended, by analytic continuation, to
a global function. For details, see \cite{e1}.

We can {\it label} the different determinations of a function of $\Res_\NN$ as
follows.
Let $\hat \varphi (\zeta)\in\Res_\NN$ and
$(\varepsilon_1, \dots , \varepsilon_n)\in \mathcal{E}$ be
a sequence of $n$ plus or minus signs.
For $\zeta$ in $\rbrack n , n+1 \lbrack$,
we will denote by $\hat \varphi^{\varepsilon_1, \dots, \varepsilon_n} (\zeta)$
the analytic continuation of $ \hat \varphi$ from $0$ to $\zeta$ along the
path that follows $\mbox{\bf R}^+$ and dodges each singularity $k$
(with $1\leq k\leq n$) to the {\it right} (resp. to the {\it left})
if $\varepsilon_k = +$ (resp. $\varepsilon_k = -$).  

For example, if $\zeta \in \rbrack 4, 5 \lbrack $,
then $\hat \varphi^{+,-,-,+}(\zeta)$ is the analytic continuation of
$\hat \varphi$ along the following path:

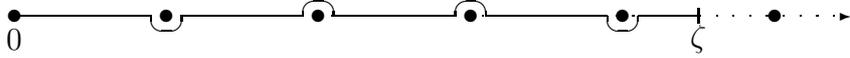
\begin{center} \unitlength=1mm
\begin{picture}(120,30) \put(5,15){\makebox(0,0){$\bullet$}}
\put(115,15){\vector(1,0){0}} \put(25,15){\makebox(0,0){$\bullet$}}
\put(45,15){\makebox(0,0){$\bullet$}}
\put(65,15){\makebox(0,0){$\bullet$}}
\put(85,15){\makebox(0,0){$\bullet$}}
\put(105,15){\makebox(0,0){$\bullet$}}
\bezier{50}(5,15)(60,15)(115,15) \put(5,15){\line(1,0){18}}
\put(27,15){\line(1,0){16}} \put(47,15){\line(1,0){16}}
\put(67,15){\line(1,0){16}} \put(87,15){\line(1,0){8}}
\put(25,15){\oval(4,4)[b]} \put(45,15){\oval(4,4)[t]}
\put(65,15){\oval(4,4)[t]} \put(85,15){\oval(4,4)[b]}
\put(95,14){\line(0,1){2}} \put(95,12){\makebox(0,0){$\zeta$}}
\put(5,12){\makebox(0,0){0}}
\end{picture}
\end{center}

Of course, $\hat \varphi^{\emptyset}(\zeta)$ ($0 < \zeta <1$) is the
unique determination of $\hat \varphi$ on $\rbrack 0,1 \lbrack$ and, for any integer $n$, a function
$\hat \varphi $ of $\Res_\NN$
has $2^n$ possibly different determinations $\hat
\varphi^{\varepsilon_1, \dots , \varepsilon_n} (\zeta)$ {\it over} the
interval $\rbrack n , n+1 \lbrack$.

There exists an algebra of operators (alien operators) which allows to analyse the
singularities of such functions.

For $\mathbf{\varepsilon} \in \mathcal{E}$, the endomorphism
$D_{\tmmathbf{\varepsilon} \bullet}$ of $\Res_\NN$ is defined as follows.
For $\hat{\varphi} \in \Res_\NN$ and  $\zeta \in] 0, 1 [$,
\begin{equation}
\hat{\psi} (\zeta) =
(D_{\tmmathbf{\varepsilon} \bullet} \hat{\varphi})(\zeta) =
\hat{\varphi}^{\tmmathbf{\varepsilon} +} (\zeta + l (
\tmmathbf{\varepsilon} \bullet)) - \hat{\varphi}^{\tmmathbf{\varepsilon}
-} (\zeta + l ( \tmmathbf{\varepsilon} \bullet))
\end{equation}
where $l ( \tmmathbf{\varepsilon} \bullet) = n$
if $\tmmathbf{\varepsilon} = (\varepsilon_1, \dots, \varepsilon_{n -
  1})$.
We also denote by $D_{\emptyset}$ the identity map on $\Res_\NN$.

It follows from the definition that the composition of such operators
is given by 
\begin{equation}
  \forall ( \tmmathbf{a}, \tmmathbf{b}) \in \mathcal{E}^2, \quad
  D_{\tmmathbf{a \bullet}} D_{\tmmathbf{b \bullet}} = D_{\tmmathbf{b} +
  \tmmathbf{a} \bullet} - D_{\tmmathbf{b} - \tmmathbf{a} \bullet}\,,
\end{equation}
which is reminescent of (\ref{prodR}),
and there is a natural gradation $\nu$ on these operators defined by $\nu
(D_{\emptyset}) = 0$ ($l (\emptyset) = 0$) and $\nu (D_{\tmmathbf{\varepsilon}
\bullet}) = l ( \tmmathbf{\varepsilon} \bullet)$.

The fundamental theorem is
\begin{theorem}
The graded algebra of alien operators
\begin{equation}
{\bf Alien}=\oplus_{n \geq 0}
    \tmop{Vect}_{\QQ} \{D_{\tmmathbf{\varepsilon \bullet}}
    \hspace{1em} ; \hspace{1em} l ( \tmmathbf{\varepsilon} \bullet) = n\}
  \end{equation}
is a Hopf algebra (with basis $D_{\tmmathbf{\varepsilon
    \bullet}}$) for the coproduct induced by the convolution:
\begin{equation}
  {\bf Op} ( \hat{\varphi} \ast \hat{\psi})  = \sum
  {\bf Op}_{(1)} ( \hat{\varphi}) \ast {\bf Op}_{(2)}
  ( \hat{\psi}) 
\quad
\text{for ${\bf Op}\in {\bf Alien}$ and  $( \hat{\varphi}, \hat{\psi}) \in (\Res_\NN)^2$.}
\label{cop}
\end{equation}
\end{theorem}
 
The proof of this nontrivial result can be found in \cite{et1} and is
also clearly illustrated in \cite{sauzin}. 
It follows from a careful combinatorial and analytic study of the analytic
continuation of functions of $\Res_\NN$, 
which can be found in \cite{et1}. We shall only summarize some key points.

\begin{itemize}
\item[--] It is not so simple to prove that the $D_{\tmmathbf{\varepsilon
    \bullet}}$ are free. Roughly speaking, this was proved by Ecalle,
using the fact that for any linear combination of such operators, there exists
a function in  $\Res_\NN$ which is not annihilated by the action of
this linear combination.
\item[--] The construction of such functions involves the use of some
  specific elements  of ${\bf Alien}$ such as :
\begin{equation}
\begin{array}{ccc}
  \Delta^+_n & = & D_{\underbrace{+ \dots +}_{n-1} \bullet}\\
  \Delta^-_n & = & - D_{\underbrace{- \dots -}_{n-1} \bullet}\\
  \Delta_n & = &
  \displaystyle \sum_{\tmmathbf{\varepsilon} \in \mathcal{E}_{n - 1}}
  \lambda^{\tmmathbf{\varepsilon}} D_{\tmmathbf{\varepsilon \bullet}}
\end{array} \label{5}
\end{equation} 
where $\lambda^{\tmmathbf{\varepsilon}} = \frac{p!q!}{(p + q + 1) !}$
with $p$ (resp. $q$) the number of plus (resp. minus) signs in
$\tmmathbf{\varepsilon}$.
It happens that each of the three families above is a family of generators.
\item[--] The existence of a coproduct $\delta$ is proved in {\cite{et1}}. The
main idea is that $\hat{\varphi} \ast \hat{\psi}$ is defined by a path
integral in the neighbourhood of $0$.
To compute $\tmmathbf{\tmop{Op}} ( \hat{\varphi} \ast \hat{\psi})$,
the analytic continuations of $\hat{\varphi} \ast \hat{\psi}$ must be known.
But, once again, these analytic continuations can be defined as path integrals
on ``self-symmetric shrinkable paths'' (see {\cite{e1}}) and a careful
decomposition of such paths (with respect to the involved analytic
continuations of $\hat{\varphi}$ and $\hat{\psi}$) yields formula~(\ref{cop}).
Indeed, we get, for $n\geq 0$,
\begin{equation}
\delta(\Delta^+_n)=\sum_{k=0}^n \Delta^+_k \otimes \Delta^+_{n-k}\,.
\end{equation}
\end{itemize}

Given these properties of ${\bf Alien}$, it is now clear that it is
isomorphic to $\Sym$, under the identification of $S_n$ and
$\Delta^+_n$. Under this isomorphism,
$D_{\varepsilon_1,..,\varepsilon_k\bullet}$ is associated with
$R_{\varepsilon_k,..,\varepsilon_1\bullet}$ (note the reversion of the
sequence, corresponding to the anti-involution in $\Sym$). To
summarize, if $\alpha$ is the isomorphism, then
\begin{equation}
\begin{array}{rcl}
\alpha(\Delta^+_n) &=& S_n \\
\alpha(\Delta^-_n) &=& (-1)^n \Lambda_n \\
\alpha(\Delta_n) &=&  \frac{\Phi_n}{n} 
\end{array}
\end{equation} 

We have thus an explicit correspondence, and both worlds of resurgence and 
noncommutative symmetric functions can now
interact. Especially any grouplike (resp. primitive) element of
${\bf Alien}$ (or $\Sym$) provides a grouplike (resp. primitive)
element of $\Sym$ (or ${\bf Alien}$).

For example, the iterated integrals introduced in this paper were
used in the framework of real resummation.

Let us go back to the above resummation scheme and assume that we are dealing
with a formal power series $\widetilde{\varphi}$ whose Borel transform
$\hat\varphi$ is in $\Res_\NN$. In order to perfom the Laplace transform along
$\RR^+$, we need to {\it uniformize} the resurgent function. This can be done
by averaging, above each interval $] n, n+1[$ the $2^n$ analytic continuations
of $\hat\varphi$, but with many analytic and algebraic constraints, {\it
e.g.}, this averaging must preserve the algebra structure but also provide a
function whose Laplace transform on $\RR^+$ converges (see \cite{e1,em1}). 

Since there are $2^n$ determinations of $\hat\varphi$ labelled
by sequences $\varepsilon_1,..,\varepsilon_n$ of signs, we can now
understand the origin of the coefficients $m^{\mathbf{\varepsilon}}$
introduced in Section~\ref{moyennes}. Indeed, if such weights are
given by the probabilities of some random walk, they fulfill all the
algebraic and analytic properties required in real resummation theory.

\section{Complements on the coproduct of {\bf Alien}.}\label{sec:coprod}

Thanks to the product in {\bf Alien}, it can be identified (as an algebra) to
$\Sym$, and it remains to understand why the coproduct of 
\begin{equation}
\Delta_n^+=D_{\underbrace{+ \dots +}_{n-1} \bullet}
\end{equation} corresponds
to the coproduct of $S_n$, that is
\begin{equation}
\delta(\Delta^+_n)=\sum_{k=0}^n \Delta^+_k \otimes \Delta^+_{n-k}.
\end{equation}
Going back to the Laplace transform, let us illustrate how this
coproduct appears.

\subsection{The Laplace Transform}

Assuming that all integrals are well-defined and
convergent for $z$ large enough, let us consider, for a given sign
$\varepsilon=\pm$ the Laplace transform on a half-line going from
$0$ to
infinity ``on the same side as 
$\varepsilon$'', precisely in the direction
$\arg(\zeta)=-\varepsilon\cdot\alpha$ where $\alpha>0$ is small
enough:
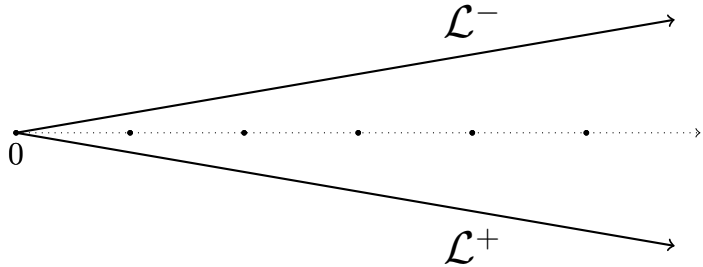
\begin{center}
\begin{tikzpicture}[scale=1.5]
\draw[dotted,->](0,0)--(6,0);
\foreach\i in{0,1,...,5}{\draw[fill](\i,0) circle (0.02);};
\node[below](z)at(0,0){$0$};

\draw[thick,->](0,0) -- (5.775,-1);
\node[font=\Large](lp)at(4,-1){$\mathcal{L}^+$};

\foreach\i in{0,1,...,5}{\draw[fill](\i,0) circle (0.02);};
\node[below](z)at(0,0){$0$};
\draw[thick,->](0,0) -- (5.775,1);
\node[font=\Large](lm)at(4,1){$\mathcal{L}^-$};

\end{tikzpicture}
\end{center}
Thanks to the definition of the convolution in $\Res_\NN$ (and to
Fubini's theorem),
\begin{equation}
\begin{array}{rcl}
\mathcal{L}^+(\hat \varphi_1 \ast \hat
\varphi_2)(z) &=& \displaystyle \int_0^{e^{-i\alpha}\infty=\infty^+}(\hat \varphi_1 \ast \hat
\varphi_2)(\zeta) e^{-z \zeta} d\zeta\\
 &=& \displaystyle \int_0^{\infty^+} \int_0^{\zeta}\hat
 \varphi_1(\zeta_1) \hat
\varphi_2(\zeta-\zeta_1) d\zeta_1 e^{-z \zeta} d\zeta \\
&=& \displaystyle \int_0^{\infty^+} \int_0^{\zeta}\hat
 \varphi_1(\zeta_1) e^{-z \zeta _1}\hat
\varphi_2(\zeta-\zeta_1) d\zeta_1 e^{-z (\zeta-\zeta_1)} d\zeta \\
&=& \displaystyle \left (\int_0^{\infty^+} \hat
 \varphi_1(\zeta_1) e^{-z \zeta _1} d\zeta_1\right )\left
 (\int_0^{\infty^+} \hat
\varphi_2(\zeta_2)  e^{-z \zeta_2} d\zeta_2 \right ) 
\\
&=& \mathcal{L}^+(\hat \varphi_1)\, \cdot\, \mathcal{L}^+(\hat \varphi_2)
\end{array}
\end{equation}
and the same holds for $\mathcal{L}^-$.

In order to compare these two Laplace transforms let us try to deform
the path defining $\mathcal{L}^+$ so that it goes to infinity in the
upper half plane. If we push the path without going through the
singularities in $\NN^{\ast}$, then, thanks to the Cauchy integral
theorem, the function obtained after summation remains the same. For
example, in the following picture :
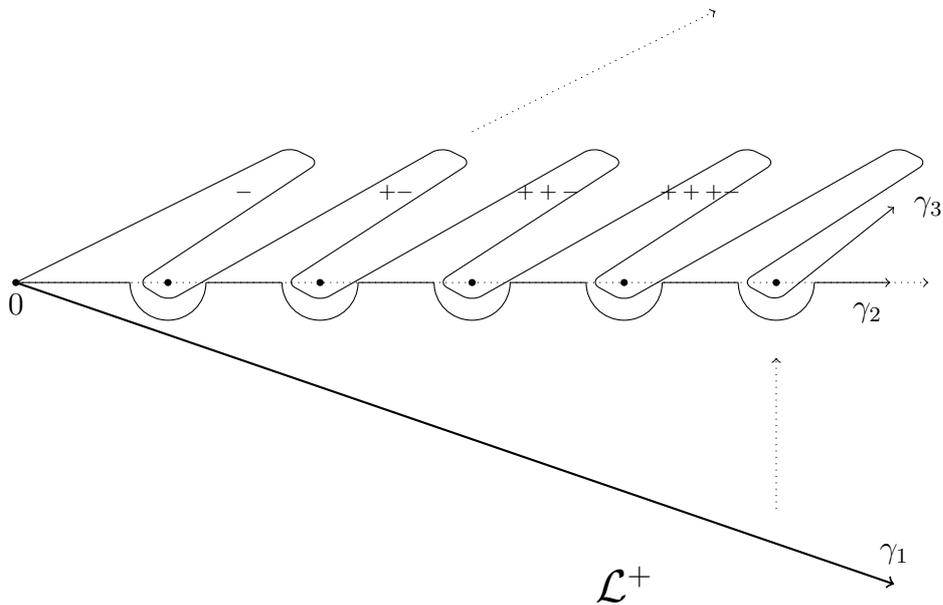
\begin{center}
\begin{tikzpicture}[scale=2]
\draw[dotted,->](0,0)--(6,0);

\foreach\i in{0,1,...,5}{\draw[fill](\i,0) circle (0.02);};
\node[below](z)at(0,0){$0$};
\draw[thick,->](0,0) -- (5.775,-2);
\node[font=\Large](lp)at(4,-2){$\mathcal{L}^+$};
\node(1)at(5.775,-1.8){$\gamma_1$};
\draw(0,0)--(0.75,0);
\draw(2.25,0)--+(.5,0);
\draw(3.25,0)--+(.5,0);
\draw(4.25,0)--+(.5,0);
\draw[->](5.25,0)--+(.5,0);
\draw(1.25,0)--+(.5,0);
\node(2)at(5.6,-.2){$\gamma_2$};
\draw(0.75,0) arc (-180:0:0.25);
\draw(1.75,0) arc (-180:0:0.25);
\draw(2.75,0) arc (-180:0:0.25);
\draw(3.75,0) arc (-180:0:0.25);
\draw(4.75,0) arc (-180:0:0.25);

\draw[dotted,->](5,-1.5)--(5,-0.5);
\draw[dotted,->](3,1)--+(1.6,0.8);
\draw[rounded
corners=4pt,->](0,0)--(1.8,0.9)--(2,0.8)--(0.8,0)--(1,-.125)--(2.8,0.9)--(3,0.8)--(1.775,0)--(2,-.125)--(3.8,0.9)--(4,0.8)--(2.775,0)--(3,-.125)--(4.8,0.9)--(5,0.8)--(3.775,0)--(4,-.125)--(5.8,0.9)--(6,0.8)--(4.775,0)--(5,-.125)--(5.775,.5);
\node(3)at(6,.5){$\gamma_3$};
\node[font=\tiny](m)at(1.5,.6){$-$};
\node[font=\tiny](pm)at(2.5,.6){$+-$};
\node[font=\tiny](ppm)at(3.5,.6){$++-$};
\node[font=\tiny](pppm)at(4.5,.6){$+++-$};
\end{tikzpicture}
\end{center}
we have
\begin{equation}
\mathcal{L}^+(\hat \varphi)(z)=\int_{\gamma_1}\hat
\varphi(\zeta)e^{-z\zeta}d\zeta=\int_{\gamma_2}\hat
\varphi(\zeta)e^{-z\zeta}d\zeta=\int_{\gamma_3}\hat
\varphi(\zeta)e^{-z\zeta}d\zeta\,.
\end{equation}
We have written on the path $\gamma_3$ the determinations of $\hat
\varphi$ that are involved when integrating along $\gamma_3$. If we
stretch this path to infinity in the direction
$e^{+i\alpha}\infty=\infty^-$, then we get the following picture:
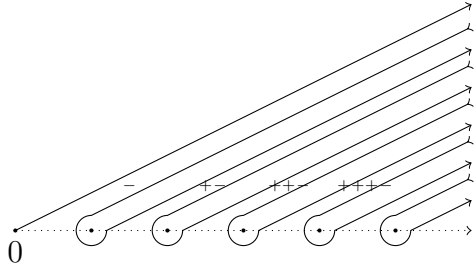
\begin{center}
\begin{tikzpicture}
\draw[dotted,->](0,0)--(6,0);

\foreach\i in{0,1,...,5}{\draw[fill](\i,0) circle (0.02);};
\node[below](z)at(0,0){$0$};
\draw[->](0,0) -- (6,3);
\draw[>->](6,2.7)--(1,0.2) arc (-270:0:0.2) --(6,2.4);
\draw[>->](6,2.2)--(2,0.2) arc (-270:0:0.2) --(6,1.9);
\draw[>->](6,1.7)--(3,0.2) arc (-270:0:0.2) --(6,1.4);
\draw[>->](6,1.2)--(4,0.2) arc (-270:0:0.2) --(6,.9);
\draw[>->](6,0.7)--(5,0.2) arc (-270:0:0.2) --(6,0.4);

\node[font=\tiny](m)at(1.5,.6){$\scriptscriptstyle -$};
\node[font=\tiny](pm)at(2.6,.6){$\scriptscriptstyle +-$};
\node[font=\tiny](ppm)at(3.6,.6){$\scriptscriptstyle++-$};
\node[font=\tiny](pppm)at(4.6,.6){$\scriptscriptstyle+++-$};

\end{tikzpicture}
\end{center}

The first path on the half-line from $0$ corresponds to
$\mathcal{L}^-(\hat \varphi)$. For the second integral $I_1$ (that goes around
$1$), since the functions are integrable at the singularities, the
circle around 1 can be shrinked and then, it is clear that one
integrates from $\infty^-$ to $1$ the determination $\varphi^-$ of
$\varphi$ then from $1$ to $\infty^-$ the determination $\varphi^+$ of
$\varphi$:
\begin{equation}
I_1(\hat \varphi)=\int_1^{\infty^-}(\hat \varphi^+ - \hat
\varphi^-)(\zeta)e^{-z\zeta} d\zeta\,.
\end{equation}
Changing $\zeta$ into $\zeta -1$, we get
\begin{equation}
I_1(\hat \varphi)=e^{-z}\int_0^{\infty^-}(\hat \varphi^+ - \hat
\varphi^-)(\zeta+1)e^{-z\zeta} d\zeta=e^{-z}
\mathcal{L}^-(\Delta^+_1\hat \varphi)(z)\,.
\end{equation}
In the same way, for the path that goes around $2$, the
corresponding integral is
\begin{equation}
\begin{array}{rcl}
I_2(\hat \varphi) &=&\displaystyle  \int_2^{\infty^-}(\hat \varphi^{++} - \hat
\varphi^{+-})(\zeta)e^{-z\zeta} d\zeta \\
&=& \displaystyle  e^{-2z}\int_0^{\infty^-}(\hat \varphi^{++} - \hat
\varphi^{+-})(\zeta+2)e^{-z\zeta} d\zeta \\
&=&\displaystyle   e^{-2z}
\mathcal{L}^-(\Delta^+_2\hat \varphi)(z)\,.
\end{array}
\end{equation}
Finally, we get
\begin{equation}
\mathcal{L}^+(\hat{\varphi})=\mathcal{L}^-(\hat{\varphi})+\sum_{k\geq
  1}e^{-k z} \mathcal{L}^-(\Delta_k^+(\hat{\varphi}))\,.
\end{equation}
If we combine this with the action of the Laplace transforms
$\mathcal{L}^+$ and $\mathcal{L}^-$, then 
\begin{equation}
\begin{array}{rcl}
\mathcal{L}^+(\hat \varphi_1 \ast \hat
\varphi_2) & = & \displaystyle \mathcal{L}^-(\hat \varphi_1 \ast \hat
\varphi_2) +\sum_{n\geq
  1}e^{-n z} \mathcal{L}^-(\Delta_n^+(\hat \varphi_1 \ast \hat
\varphi_2)) \\
&=& \displaystyle \mathcal{L}^+(\hat \varphi_1).\mathcal{L}^+(\hat
\varphi_2) \\
 &=& \displaystyle \left (\mathcal{L}^-(\hat \varphi_1 ) +\sum_{k\geq
  1}e^{-k z} \mathcal{L}^-(\Delta_k^+(\hat \varphi_1 )) \right ) \\
 && \times \displaystyle \left (\mathcal{L}^-(\hat \varphi_2 ) +\sum_{l\geq
  1}e^{-l z} \mathcal{L}^-(\Delta_l^+(\hat \varphi_2 )) \right ) \\
&=& \displaystyle \mathcal{L}^-(\hat \varphi_1 \ast \hat
\varphi_2) +\sum_{n\geq
  1}e^{-n z} \sum_{k+l=n}\mathcal{L}^- (\Delta_k^+(\hat \varphi_1)
\ast \Delta_l^+ \hat
\varphi_2)) 
\end{array}
\end{equation}
and  the coefficient of   $e^{-n z}$ is precisely
given by the proposed coproduct formula for
$\Delta^+_n$. 
The actual proof of the existence of  this coproduct is also based on path
deformation.  We will illustrate it in the following subsection.

\subsection{Path deformation and coproduct}

In the definition of $\Res_\NN$, the convolution was defined in the
neighbourhood of $0$ by the path integral: 
\begin{equation}
\hat \varphi_3(\zeta) =(\hat \varphi_1 \ast \hat
\varphi_2)(\zeta)=\int_0^{\zeta} \hat \varphi_1(\zeta_1) \, \hat
\varphi_2(\zeta - \zeta_1) \, d \zeta_1 \; \; (0<\zeta < 1)\,,
\end{equation}
where $\hat \varphi_1,\hat \varphi_2 \in \Res_\NN$

In order to let $\Delta_n^+$ act on the convolution product, the germ $\hat
\varphi_3(\zeta)$ must be extended by analytic continuation, and, 
as this germ is defined as a path integral, the continuation of the
germ is obtained by deformation of the path defining the
convolution. But this deformation must be done carefully since one has
to avoid the singularities of $\hat \varphi_1(\zeta_1)$ but also the
singularities of $\hat
\varphi_2(\zeta - \zeta_1)$, namely the set $\{\zeta -n,n\geq 1 \}$. 
Morever, in order to respect the
commutativity of the convolution product, we have to take a self-symmetric path of
analytic continuation from $0$ to $\zeta$, that is path such that, if
$\zeta_1$ is on the path, $\zeta - \zeta_1$ is also on the path. 

In order to do so, we can apply the following procedure (see
\cite{et1}): Starting 
from $\zeta$ near $0$, we deform the path to get the attempted
analytic continuation, without going through the singularities in
$\NN^{\ast}$ {\it and} $\{\zeta -n,n\geq 1 \}$. So we draw, 
these sets 
for a given $\zeta$, 
and try to deform the path.
For example, let us compute
\begin{equation}
\Delta_2(\hat \varphi_3)(\zeta)= (\hat \varphi_3^{++} -   \hat
\varphi_3^{+-})(\zeta+2)\,.
\end{equation}
To do so, we need to know $\hat \varphi_3^{++}(\zeta)$ and $\hat
\varphi_3^{+-}(\zeta)$ for $\zeta \in]2,3[$.
Assuming $\hat \varphi_3$ in $Res_{\NN}$, $\hat \varphi_3^{++}$
is obtained by deformation of paths, starting from $\zeta$ near $0$:
\begin{center}
\begin{tikzpicture}[scale=2.5]
\draw[dotted,->](0,0)--(4,0);

\foreach\i in{0,1,...,3}{\draw[fill](\i,0) circle (0.02);};
\node[above](0)at(0,0){$0$};
\node[above](1)at(1,0){$1$};
\node[above](2)at(2,0){$2$};
\node[above](3)at(3,0){$3$};
\draw[-|](0,0)--(1,-0.5);
\node[right](z)at(1,-0.5){$\zeta$};
\draw[dashed,->](1.3,-0.5)--(2.3,-0.5);
\end{tikzpicture}

\begin{tikzpicture}[scale=2.5]
\draw[dotted,->](0,0)--(4,0);

\foreach\i in{0,1,...,3}{\draw[fill](\i,0) circle (0.02);};
\node[above](0)at(0,0){$0$};
\node[above](1)at(1,0){$1$};
\node[above](2)at(2,0){$2$};
\node[above](3)at(3,0){$3$};
\draw[-|](0,0)--(2,-0.5);
\node[right](z)at(2,-0.5){$\zeta$};
\draw[dashed,->](2,-0.4)--(2.3,-0.2);
\end{tikzpicture}

\begin{tikzpicture}[scale=2.5]
\draw[dotted,->](0,0)--(4,0);

\foreach\i in{0,1,...,3}{\draw[fill](\i,0) circle (0.02);};
\node[above](0)at(0,0){$0$};
\node[above](1)at(1,0){$1$};
\node[above](2)at(2,0){$2$};
\node[above](3)at(3,0){$3$};
\draw[-|](0,0)--(0.95,0) arc (-180:0:0.05) -- (1.1,0)--(1.95,0) arc
(-180:0:0.05) --(2.4,0);
\node[below](z)at(2.4,0){$\zeta$};

\end{tikzpicture}

\end{center}
Since $ \varphi_3(\zeta)$ is given by a convolution integral, we must deform 
the path of analytic continuation in a self-symmetric way and avoid the
singularities over $\NN^{\ast}$ and over their symmetrics $\zeta -\NN^{\ast}$.
If we draw these singularities, we get the following path of analytic
continuation, which  gives $ \varphi_3(\zeta)$ as an integral:

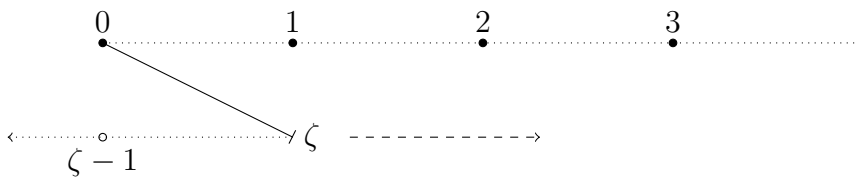
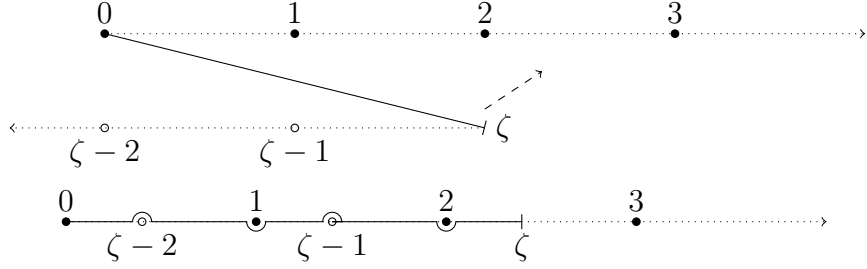
\begin{center}
\begin{tikzpicture}[scale=2.5]
\draw[dotted,->](0,0)--(4,0);

\foreach\i in{0,1,...,3}{\draw[fill](\i,0) circle (0.02);};
\node[above](0)at(0,0){$0$};
\node[above](1)at(1,0){$1$};
\node[above](2)at(2,0){$2$};
\node[above](3)at(3,0){$3$};
\draw[-|](0,0)--(1,-0.5);
\node[right](z)at(1,-0.5){$\zeta$};
\draw[dotted,->](1,-0.5)--(-0.5,-0.5);
\draw(0,-0.5) circle (0.02);
\node[below](z1)at(0,-0.5){$\zeta-1$};

\draw[dashed,->](1.3,-0.5)--(2.3,-0.5);
\end{tikzpicture}

\begin{tikzpicture}[scale=2.5]
\draw[dotted,->](0,0)--(4,0);

\foreach\i in{0,1,...,3}{\draw[fill](\i,0) circle (0.02);};
\node[above](0)at(0,0){$0$};
\node[above](1)at(1,0){$1$};
\node[above](2)at(2,0){$2$};
\node[above](3)at(3,0){$3$};
\draw[-|](0,0)--(2,-0.5);
\node[right](z)at(2,-0.5){$\zeta$};

\draw[dotted,->](2,-0.5)--(-0.5,-0.5);
\draw(0,-0.5) circle (0.02);
\draw(1,-0.5) circle (0.02);

\node[below](z1)at(1,-0.5){$\zeta-1$};
\node[below](z2)at(0,-0.5){$\zeta-2$};

\draw[dashed,->](2,-0.4)--(2.3,-0.2);
\end{tikzpicture}

\begin{tikzpicture}[scale=2.5]
\draw[dotted,->](0,0)--(4,0);

\foreach\i in{0,1,...,3}{\draw[fill](\i,0) circle (0.02);};
\node[above](0)at(0,0){$0$};
\node[above](1)at(1,0){$1$};
\node[above](2)at(2,0){$2$};
\node[above](3)at(3,0){$3$};
\draw[-|](0,0)--(0.35,0) arc (180:0:0.05) --(0.95,0) arc (-180:0:0.05)--(1.35,0) arc (180:0:0.05)  -- (1.4,0)--(1.95,0) arc
(-180:0:0.05) --(2.4,0);
\node[below](z)at(2.4,0){$\zeta$};

\draw(0.4,0) circle (0.02);
\draw(1.4,0) circle (0.02);

\node[below](z1)at(1.4,0){$\zeta-1$};
\node[below](z2)at(0.4,0){$\zeta-2$};
\end{tikzpicture}
\end{center}

For $\hat \varphi_3^{+-}(\zeta)$ ($\zeta \in]2,3[$),
the natural way to get $\hat \varphi_3^{++}$
is obtained by deformation, starting from $\zeta$ near $0$:

\begin{center}
\begin{tikzpicture}[scale=2.5]
\draw[dotted,->](0,0)--(4,0);

\foreach\i in{0,1,...,3}{\draw[fill](\i,0) circle (0.02);};
\node[above](0)at(0,0){$0$};
\node[above](1)at(1,0){$1$};
\node[above](2)at(2,0){$2$};
\node[above](3)at(3,0){$3$};
\draw[-|](0,0)--(1,-0.5);
\node[right](z)at(1,-0.5){$\zeta$};
\draw[dashed,->](1,-0.3)--(1.5,0.3);
\end{tikzpicture}

\begin{tikzpicture}[scale=2.5]
\draw[dotted,->](0,0)--(4,0);

\foreach\i in{0,1,...,3}{\draw[fill](\i,0) circle (0.02);};
\node[above](0)at(0,0){$0$};
\node[above](1)at(1,0){$1$};
\node[above](2)at(2,0){$2$};
\node[above](3)at(3,0){$3$};
\draw[-|](0,0)--(0.95,0) arc (-180:0:0.05)--(2,0.5);
\node[right](z)at(2,0.5){$\zeta$};
\draw[dashed,->](2,0.4)--(2.3,0.1);
\end{tikzpicture}

\begin{tikzpicture}[scale=2.5]
\draw[dotted,->](0,0)--(4,0);

\foreach\i in{0,1,...,3}{\draw[fill](\i,0) circle (0.02);};
\node[above](0)at(0,0){$0$};
\node[above](1)at(1,0){$1$};
\node[above](2)at(2,0){$2$};
\node[above](3)at(3,0){$3$};
\draw[-|](0,0)--(0.95,0) arc (-180:0:0.05) -- (1.1,0)--(1.95,0) arc
(180:0:0.05) --(2.4,0);
\node[below](z)at(2.4,0){$\zeta$};
\end{tikzpicture}
\end{center}
Once again, since $ \varphi_3(\zeta)$ is given by a convolution integral, we
must deform the path of analytic continuation in a self-symmetric way:

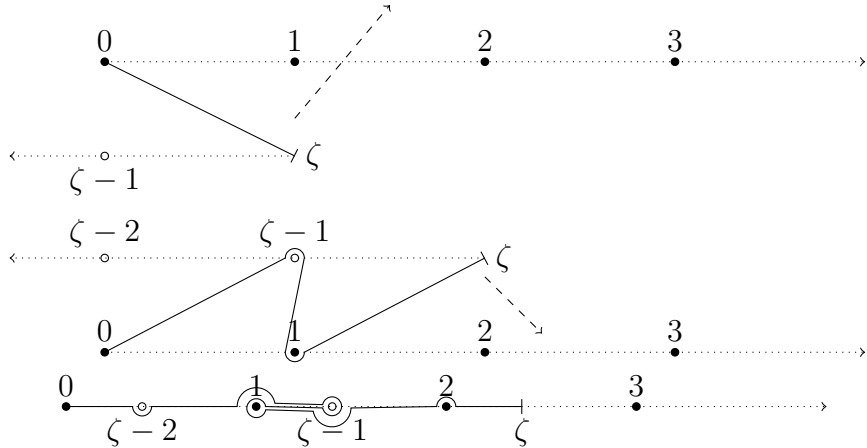
\begin{center}
\begin{tikzpicture}[scale=2.5]
\draw[dotted,->](0,0)--(4,0);

\foreach\i in{0,1,...,3}{\draw[fill](\i,0) circle (0.02);};
\node[above](0)at(0,0){$0$};
\node[above](1)at(1,0){$1$};
\node[above](2)at(2,0){$2$};
\node[above](3)at(3,0){$3$};
\draw[-|](0,0)--(1,-0.5);
\node[right](z)at(1,-0.5){$\zeta$};
\draw[dotted,->](1,-0.5)--(-0.5,-0.5);
\draw(0,-0.5) circle (0.02);
\node[below](z1)at(0,-0.5){$\zeta-1$};
\draw[dashed,->](1,-0.3)--(1.5,0.3);
\end{tikzpicture}

\begin{tikzpicture}[scale=2.5]
\draw[dotted,->](0,0)--(4,0);

\foreach\i in{0,1,...,3}{\draw[fill](\i,0) circle (0.02);};
\node[above](0)at(0,0){$0$};
\node[above](1)at(1,0){$1$};
\node[above](2)at(2,0){$2$};
\node[above](3)at(3,0){$3$};

\draw[-|](0,0)--(0.95,0.5) arc (180:0:0.05)--(0.95,0) arc (-180:0:0.05)--(2,0.5);
\node[right](z)at(2,0.5){$\zeta$};

\draw[dotted,->](2,0.5)--(-0.5,0.5);
\draw(0,0.5) circle (0.02);
\draw(1,0.5) circle (0.02);

\node[above](z1)at(1,0.5){$\zeta-1$};
\node[above](z2)at(0,0.5){$\zeta-2$};

\draw[dashed,->](2,0.4)--(2.3,0.1);
\end{tikzpicture}

\begin{tikzpicture}[scale=2.5]
\draw[dotted,->](0,0)--(4,0);

\foreach\i in{0,1,...,3}{\draw[fill](\i,0) circle (0.02);};
\node[above](0)at(0,0){$0$};
\node[above](1)at(1,0){$1$};
\node[above](2)at(2,0){$2$};
\node[above](3)at(3,0){$3$};
\draw[-|](0,0)--(0.35,0) arc (-180:0:0.05) --(0.90,0)
arc(180:10:0.1)-- (1.35,0.01) arc (170:-170:0.05)  -- (1.05,0) arc
(10:350:0.05)-- (1.3,-0.025) arc(-170:0:0.1) --(1.95,0) arc
(180:0:0.05) --(2.4,0);
\node[below](z)at(2.4,0){$\zeta$};

\draw(0.4,0) circle (0.02);
\draw(1.4,0) circle (0.02);

\node[below](z1)at(1.4,0){$\zeta-1$};
\node[below](z2)at(0.4,0){$\zeta-2$};
\end{tikzpicture}
\end{center}

For these symmetric paths, we can shrink the different
circles and, using the integral expression of $\hat \varphi_3$ when
$\zeta_1$ runs along the path, we can mark the determination of $\hat
\varphi_1 (\zeta_1)$. Since the path is symmetric, the
determination of $\hat \varphi_2(\zeta -\zeta_1)$ is given by
symmetry.
For $\hat \varphi_3^{++}$ the information is summarized in the following table

\begin{center}
\begin{tabular}{|c|ccccccccccc|}
\hline
$\zeta_1$ & $0$ & $\rightarrow$ & $\zeta-2$ & $\rightarrow$ & $1$  &
$\rightarrow$ & $\zeta-1$ &  $\rightarrow$ & $2$  & $\rightarrow$
& $\zeta$ \\
\hline
$\hat
\varphi_1 (\zeta_1)$ & \vline &$ \emptyset$ & \vline & $\emptyset$ & \vline & $+$ & \vline &
$+$ & \vline & $++$ & \vline \\
\hline
$\hat \varphi_2(\zeta -\zeta_1)$& \vline &$++$ & \vline & $+$ & \vline & $+$ & \vline &
$\emptyset$ & \vline & $\emptyset$ & \vline \\
\hline
\end{tabular}
\end{center}
and for $\hat \varphi_3^{+-}$: 
\begin{center}
\begin{tabular}{|c|ccccccccccccccc|}
\hline
$\zeta_1$ & $0$ & $\rightarrow$ & $\zeta-2$ & $\rightarrow$ & $1$  &
$\rightarrow$ & $\zeta-1$ &  $\rightarrow$ & $1$ &$\rightarrow$ &
$\zeta-1$ &$\rightarrow$ & $2$  & $\rightarrow$
& $\zeta$ \\
\hline
$\hat
\varphi_1 $ & \vline &$ \emptyset$ & \vline & $\emptyset$ & \vline & $-$ &
\vline & $-$ & \vline & $+$ & \vline &$+$ &\vline & $+-$ &\vline \\
\hline
$\hat \varphi_2$& \vline &$+-$ & \vline & $+$ & \vline & $+$ & \vline
& $-$ & \vline & $-$ & \vline & $\emptyset$ & \vline & $\emptyset$ & \vline \\
\hline
\end{tabular}
\end{center}
If we compute carefully the convolution integral defining the
difference 
\begin{equation}
(\hat
\varphi_3^{++} - \hat \varphi_3^{+-})(\zeta)
 = (\Delta_2^+\hat \varphi_3)(\zeta-2)\quad  (\zeta \in(2,3)),
\end{equation}
some cancellations occur and we get :
\begin{equation}
\begin{array}{rcl}
(\hat \varphi_3^{++} - \hat \varphi_3^{+-})(\zeta)
 &=& \displaystyle \int_0^{\zeta-2}
\hat \varphi_1^{\emptyset}(\zeta_1)(\hat \varphi_2^{++}-\hat
\varphi_2^{+-})(\zeta-\zeta_1) d\zeta_1 \\
 & & +\displaystyle \int_1^{\zeta-1}
(\hat \varphi_1^{+}-\hat \varphi_1^{-})(\zeta_1)(\hat \varphi_2^{+}-\hat
\varphi_2^{-})(\zeta-\zeta_1) d\zeta_1 \\
 & & +\displaystyle \int_2^{\zeta}
(\hat \varphi_1^{++}-\hat
\varphi_1^{+-})(\zeta_1)\hat \varphi_2^{\emptyset}((\zeta-\zeta_1) d\zeta_1\,.
\end{array}
\end{equation}
If $\zeta =\xi+2$ ($ \xi \in]0,1[$), then, by translation of the
variable in each integral, 

\begin{equation}
\begin{array}{rcl}
\Delta_2^+ (\hat \varphi_3)(\xi) &=& \Delta_2^+ (\hat \varphi_1 \ast
\hat \varphi_2)(\xi) \\
 &= &(\hat \varphi_3^{++} - \hat
\varphi_3^{+-})(\xi +2) \\
 &=& \displaystyle \int_0^{\xi}
\hat \varphi_1^{\emptyset}(\zeta_1)(\hat \varphi_2^{++}-\hat
\varphi_2^{+-})(\xi-\zeta_1+2) d\zeta_1 \\
 & & +\displaystyle \int_0^{\xi}
(\hat \varphi_1^{+}-\hat \varphi_1^{-})(\zeta_1+1)(\hat \varphi_2^{+}-\hat
\varphi_2^{-})(\xi-\zeta_1+1) d\zeta_1 \\
 & & +\displaystyle \int_0^{\xi}
(\hat \varphi_1^{++}-\hat
\varphi_1^{+-})(\zeta_1+2)\hat \varphi_2^{\emptyset}((\xi-\zeta_1) d\zeta_1\,.
\end{array}
\end{equation}
This is precisely the expected result since, near the origin, 
\begin{equation}
\Delta_2^+ (\hat \varphi_1 \ast
\hat \varphi_2)=\hat \varphi_1\ast (\Delta_2^+ \hat \varphi_2)+
(\Delta_1^+ \hat \varphi_1)\ast (\Delta_1^+ \hat \varphi_2)) +
(\Delta_2^+ \hat \varphi_1) \ast \hat \varphi_2\,.
\end{equation}

This way of computing the analytic continuations of a convolution product can
be shown to yield in all cases the claimed coproduct of the operators
$\Delta_{n}^+$.


\end{document}